\documentclass[12pt,a4paper]{amsart}
\usepackage{amsmath, amssymb, amsthm, yhmath}
\input cyracc.def

\usepackage{graphicx}
\usepackage{caption}  
\usepackage{float}    
\usepackage{mathtools} 
\usepackage[colorlinks=false,
            linkbordercolor={0 0 1},   
            citebordercolor={0 1 0},   
            urlbordercolor={1 0 0}     
]{hyperref}
\numberwithin{equation}{section}
\usepackage{enumerate}
\newtheorem{theorem}{Theorem}[section]
\newtheorem{lemma}[theorem]{Lemma}
\newtheorem{claim}{Claim}[section]
\newtheorem{corollary}[theorem]{Corollary}
\newtheorem{proposition}[theorem]{Proposition}

\theoremstyle{definition} 

\newtheoremstyle{remarkstyle} 
   {}   
   {}   
  {\normalfont}  
  {}      
  {\itshape}  
  {.}     
  { }     
  {}      
\theoremstyle{remarkstyle}
\newtheorem*{remark}{Remark} 
\newtheorem{problem}{Problem}
\allowdisplaybreaks[4]

\begin{document}

\author[K. Matsuzaki]{Katsuhiko Matsuzaki\textsuperscript{$1$}} 
\address{$1.$ Department of Mathematics, School of Education, Waseda University, Tokyo 169-8050, Japan} 
\email{matsuzak@waseda.jp} 

\author[F. Tao]{Fei Tao\textsuperscript{$2$}} 
\address{$2.$ Beijing International Center for Mathematical Research, Peking University, Beijing 100871, P. R. China} 
\email{ferrytau@pku.edu.cn}

\thanks{The first author$^1$ is partially supported by Japan Society for the Promotion of Science (KAKENHI 23K25775 and 23K17656); the second author$^2$ is partially supported by National Natural Science Foundation of China (No.12131009 and No.12326601)}

\subjclass[2020]{Primary 30C62,30H35,30H30; Secondary 46E15, 53A04}

\title[Characterization of Asymptotically Smooth Chord-arc Curves]{Characterization of Asymptotically Smooth Chord-arc Curves}
\keywords{Asymptotically Smooth Curves, Asymptotically Conformal Curves, Chord-arc Curves, $\mathrm{BMOA}$, $\mathrm{VMOA}$, (Little) Bloch Functions}
\begin{abstract}
We construct an explicit example of an asymptotically conformal chord-arc curve that fails to be asymptotically smooth. 
This implies that a function belonging to both the little Bloch space and $\mathrm{BMOA}$ does not necessarily lie in $\mathrm{VMOA}$, and that a strongly quasisymmetric homeomorphism which is symmetric is not necessarily strongly symmetric. 
We also provide a complete characterization of asymptotically smooth curves in terms of asymptotic conformality and uniform approximability. 
\end{abstract}
\maketitle

\section{Introduction}

The study of conformal maps from the unit disk $\mathbb{D}$ onto bounded Jordan domains $\Omega$ has long been a central theme in geometric function theory. 
A classical result due to Beurling and Ahlfors \cite{BA} shows that such a conformal map $f$ admits a quasiconformal extension to the exterior disk $\mathbb{D}^*=\widehat{\mathbb{C}}\setminus \overline{\mathbb{D}}$ if and only if
the boundary homeomorphism of the unit circle $\mathbb{S}=\partial\mathbb{D}$, induced by conformal welding, is precisely a quasisymmetric self-homeomorphism of $\mathbb{S}$.

A homeomorphism $h\colon\mathbb{S}\to\mathbb{S}$ is called \emph{quasisymmetric} if there exists a constant $M\geq 1$ such that for any $t>0$, $x\in[0,\,2\pi]$, one has
\[
\frac{1}{M} \;\leq\; \left|\frac{h(e^{i(x+t)})-h(e^{ix})}{h(e^{ix})-h(e^{i(x-t)})} \right|\;\leq\; M.
\]
The set of all quasisymmetric self-homeomorphisms of $\mathbb{S}$ is denoted by $\mathrm{QS}(\mathbb{S})$.

By conformal welding, every $h \in \mathrm{QS}(\mathbb{S})$ can be represented as $h = g^{-1} \circ f|_{\mathbb{S}}$, where $f\colon\mathbb{D} \to \Omega$ and $g\colon\mathbb{D}^* \to \Omega^*$ are conformal maps onto the complementary Jordan domains $\Omega$ and $\Omega^*$, respectively, with $\Omega$ a bounded quasidisk.

The universal Teichm\"uller space $T$ is identified with the group $\mathrm{QS}(\mathbb{S})$ modulo $\mathrm{M\ddot ob}(\mathbb{S})$, the group of M\"obius transformations of $\mathbb{S}$ (see \cite{Le}). Geometrically, $T$ can be regarded as the collection of normalized quasicircles, namely the images of $\mathbb{S}$ under conformal maps $f$ that admit a quasiconformal extension to $\mathbb{D}^*$. Analytically, it can be described as the set of Bloch functions given by $\varphi=\log f'$ of the corresponding conformal maps. This correspondence has provided a rich interplay between function spaces, circle homeomorphisms, and the geometry of Jordan curves.

A holomorphic function $\varphi$ on the unit disk $\mathbb{D}$ is called a \emph{Bloch function} if it satisfies 
\[
\Vert \varphi \Vert_B\coloneqq\sup_{z \in \mathbb{D}}(1-|z|^2)|\varphi'(z)|<\infty.
\]
The Banach space of all Bloch functions with this norm (modulo additive constants) is denoted by $B(\mathbb{D})$. 
If $\varphi \in B(\mathbb{D})$ vanishes at the boundary in the sense that 
\[
\lim_{|z|\to 1}(1-|z|^2)|\varphi'(z)|=0
\]
uniformly, then $\varphi$ is called a \emph{little Bloch function}. 
The collection of all such functions forms a closed subspace of $B(\mathbb{D})$, denoted by $B_0(\mathbb{D})$. 
This little class contains many other vanishing classes of holomorphic functions.

Motivated by replacing quasiconformal homeomorphisms with bi-Lipschitz homeomorphisms of the Euclidean plane $\mathbb{C}$, a parallel theory was developed by Astala and Zinsmeister \cite{AZ}. 
The corresponding function spaces are defined as follows. 

A holomorphic function $\varphi$ on $\mathbb{D}$ belongs to $\mathrm{BMOA}(\mathbb{D})$ if it lies in the Hardy space $H^2(\mathbb{D})$ and its boundary value function $\varphi|_{\mathbb{S}}$ belongs to $\mathrm{BMO}$. 
The function $\varphi$ is in the little subspace $\mathrm{VMOA}(\mathbb{D})$ if $\varphi|_{\mathbb{S}}$ lies in $\mathrm{VMO}$. 
These spaces can also be characterized in terms that $(1-|z|^2)|\varphi'(z)|^2\,dxdy$ is a Carleson measure, and a vanishing Carleson measure, on $\mathbb{D}$ (see Section~\ref{Teich} for detailed definitions), respectively. Provided a suitable $\mathrm{BMOA}$ norm, these spaces are regarded as Banach spaces.

It is known that $\mathrm{VMOA}(\mathbb{D})$ is contained in $B_0(\mathbb{D})$ (see \cite[Corollary 5.2]{Gi}), meaning that the vanishing oscillation condition in the $\mathrm{BMOA}$ norm implies the corresponding vanishing condition in the Bloch norm. 
Since $\mathrm{VMOA}(\mathbb{D})$ is a subspace of $\mathrm{BMOA}(\mathbb{D})$, a natural question then arises: does the vanishing condition for Bloch functions conversely imply that for $\mathrm{BMOA}$? 
Explicitly, this question asks whether $\mathrm{VMOA}(\mathbb{D})$ coincides with $\mathrm{BMOA}(\mathbb{D}) \cap B_0(\mathbb{D})$. 
However, existing studies have not detected any difference between these two classes in several important settings. 
For instance, available characterizations fail to distinguish them when $\varphi$ is given by a lacunary series $\varphi(z)=\sum_{k=1}^\infty a_k z^{n_k}\,(z \in \mathbb{D})$ with Hadamard gaps $\inf_k (n_{k+1}/n_k) >1$ (see \cite[Corollary 2]{HS}, \cite[Chap.~11]{Pom}), and also for random holomorphic functions $\varphi_\omega(z)=\sum_{n=0}^\infty a_n \omega_n z^n$, where $\omega=(\omega_n)$ is a discrete stochastic process consisting of random variables taking values in $\mathbb{S}$ (see \cite[Section 5]{NP}, \cite[Theorem 3.2]{Sledd}, \cite{Wulan}).
A further discussion on these problems can be also found in a research report \cite{MT} of the present paper.

In Section~\ref{Teich}, we provide more detailed background on the theory of Teichm\"uller spaces, including several definitions of the function spaces introduced in this section. 
This material serves as preliminaries for the concepts used in the present paper.

Given a conformal map $f$ from $\mathbb{D}$ to a bounded domain $\Omega$ admitting a quasiconformal extension, the logarithm of the derivative $f'$ belongs to the Bloch space, that is, $\varphi=\log f' \in B(\mathbb{D})$. 
The collection $\mathcal{T}$ of all such $\varphi$ forms a connected open subset of $B(\mathbb{D})$ containing the origin (see \cite{Z}), which is interpreted as the Bers fiber space over the universal Teichm\"uller space $T$ (see \cite[Appendix]{Teo}). 
We denote this projection by $p\colon\mathcal{T} \to T$.

An element $h \in \mathrm{QS}(\mathbb{S})$ is said to be \emph{symmetric} if
\[
\lim_{t \to 0}\ \frac{h(e^{i(x+t)})-h(e^{ix})}{h(e^{ix})-h(e^{i(x-t)})}=1
\] 
uniformly. 
The subgroup of all such symmetric homeomorphisms is denoted by $\mathrm{S}(\mathbb{S})$. 
Furthermore, $h \in \mathrm{QS}(\mathbb{S})$ is called \emph{strongly quasisymmetric} if it is absolutely continuous and its derivative $|h'|$ is a Muckenhoupt $A_\infty$-weight (see \cite{CF}), which in particular implies that $\log h' \in \mathrm{BMO}(\mathbb{S})$. 
It is called \emph{strongly symmetric} if $\log h' \in \mathrm{VMO}(\mathbb{S})$. 
The corresponding subgroups are denoted by $\mathrm{SQS}(\mathbb{S})$ and $\mathrm{SS}(\mathbb{S})$, respectively. 
We see that $\mathrm{SS}(\mathbb{S})$ is contained in both $\mathrm{SQS}(\mathbb{S})$ and $\mathrm{S}(\mathbb{S})$ (see Section~\ref{Teich}).

Under the identification of $T$ with $\mathrm{QS}(\mathbb{S})$ modulo $\mathrm{M\ddot ob}(\mathbb{S})$, we have that $\mathrm{S}(\mathbb{S})$ corresponds to the little universal Teich\-m\"ul\-ler space $T_0$, $\mathrm{SQS}(\mathbb{S})$ to the $\mathrm{BMO}$ Teich\-m\"ul\-ler space $T_B$, and $\mathrm{SS}(\mathbb{S})$ to the $\mathrm{VMO}$ Teich\-m\"ul\-ler space $T_V$. 
Through the projection $p\colon\mathcal{T} \to \mathrm{QS}(\mathbb{S})/\mathrm{M\ddot ob}(\mathbb{S})$, these spaces are related by
\begin{align*}
T_0&=p(B_0(\mathbb{D}) \cap \mathcal{T})=\mathrm{S}(\mathbb{S})/\mathrm{M\ddot ob}(\mathbb{S});\\ 
T_B&=p(\mathrm{BMOA}(\mathbb{D}) \cap \mathcal{T})=\mathrm{SQS}(\mathbb{S})/\mathrm{M\ddot ob}(\mathbb{S});\\ 
T_V&=p(\mathrm{VMOA}(\mathbb{D}) \cap \mathcal{T})=\mathrm{SS}(\mathbb{S})/\mathrm{M\ddot ob}(\mathbb{S}).
\end{align*}

These relationships between the subspaces of $\mathcal{T} \subset B(\mathbb{D})$ and $\mathrm{QS}(\mathbb{S})/\mathrm{M\ddot ob}(\mathbb{S})$ indicate that the question of asking whether the inclusion $\mathrm{VMOA}(\mathbb{D}) \subset \mathrm{BMOA}(\mathbb{D}) \cap B_0(\mathbb{D})$ is strict is equivalent to asking whether the inclusion $\mathrm{SS}(\mathbb{S}) \subset \mathrm{SQS}(\mathbb{S}) \cap \mathrm{S}(\mathbb{S})$ is strict. 
However, constructing an explicit element in this possible gap remains challenging. For instance, while $\log|z-1|$ on $\mathbb{S}$ lies in $\mathrm{BMO}(\mathbb{S}) \setminus \mathrm{VMO}(\mathbb{S})$, the integral of $|z-1|$ near $1$ does not yield a vanishing quasisymmetry quotient.

Let $f\colon\mathbb{D} \to \Omega$ be a conformal homeomorphism. 
To characterize $\varphi = \log f'$, we consider the following geometric properties of the bounded Jordan curve $\Gamma = \partial \Omega$, which were introduced by Pommerenke \cite{Po78}. 
For distinct points $a, b \in \Gamma$, the subarc between $a$ and $b$ with smaller diameter is denoted by $\wideparen{ab}$, and its length is denoted by $\ell(\wideparen{ab})$ when it is rectifiable.
A quasicircle $\Gamma$ is said to be \emph{asymptotically conformal} if
\[
\lim_{|a-b|\to 0} \max_{w \in \wideparen{ab}} \frac{|a-w|+|w-b|}{|a-b|}=1.
\]
A rectifiable curve $\Gamma$ is called \emph{chord-arc} if there exists $C \geq 1$ such that
$\ell(\wideparen{ab})/|a-b|\leq C$ for any $a,b \in \Gamma$, and \emph{asymptotically smooth} if
\[
\lim_{|a-b|\to 0}\frac{\ell(\wideparen{ab})}{|a-b|}=1.
\]

The geometric characterization asserts that $\varphi = \log f' \in B_0(\mathbb{D}) \cap \mathcal{T}$ if and only if $\Gamma = f(\mathbb{S})$ is asymptotically conformal, and that $\varphi \in \mathrm{VMOA}(\mathbb{D}) \cap \mathcal{T}$ 
if and only if $\Gamma$ is asymptotically smooth \cite[p. 172]{Pom}. 
In addition, $\varphi$ lies in an open subset containing the origin in $\mathrm{BMOA}(\mathbb{D})$ if $\Gamma$ is chord-arc \cite[III. Proposition 6]{Zin}; conversely, if $\varphi$ has sufficiently small $\mathrm{BMOA}$ norm, then $\Gamma$ is chord-arc \cite[Theorem 3]{Pom1}, see also \cite{Se}. 
More generally, an element of $\mathrm{BMOA}(\mathbb{D}) \cap \mathcal{T}$ is characterized by the Bishop--Jones condition on $\Gamma$ given in \cite[Theorem 4]{BJ}, which, roughly speaking, requires that from every point in the interior domain $\Omega$ bounded by $\Gamma$, a uniformly large portion of $\Gamma$ is visible along a chord-arc curve.

By reformulating the question of inclusion relations between function spaces in terms of the geometric behavior of boundary curves of conformal maps, we obtain the answer through a concrete example.

\begin{theorem}\label{main1}
There exists an asymptotically conformal chord-arc curve that is not asymptotically smooth.
\end{theorem}

\begin{corollary}\label{answer}
$(1)$ The inclusion $\mathrm{VMOA}(\mathbb{D}) \subset \mathrm{BMOA}(\mathbb{D}) \cap B_0(\mathbb{D})$ is strict.\\
$(2)$ The inclusion $\mathrm{SS}(\mathbb{S}) \subset \mathrm{SQS}(\mathbb{S}) \cap \mathrm{S}(\mathbb{S})$ is strict.
\end{corollary}

For instance, one may try to realize the gap between the classes of curves stated in Theorem~\ref{main1} by a Jordan curve obtained by appropriately connecting the endpoints of the graph of a continuous function on a bounded closed interval. 
However, such an example given in \cite[Example~4]{GGPPR} does not serve for this purpose.
Indeed, taking the graph domain $\Omega$ determined by the function $y = x^2 \cos(\pi/x^2)$ on $[-1,1]$, let $\varphi=\log g'$, where $g\colon\mathbb{D} \to \Omega$ is the conformal homeomorphism. 
Then non-rectifiability of the boundary $\partial \Omega$ implies that $\varphi \notin \mathrm{VMOA}(\mathbb{D})$, but $\varphi \in \mathrm{BMOA}(\mathbb{D})$ holds; see \cite[p.~617, line 6]{AZ}). 
The latter fact is essentially due to boundedness of $\operatorname{Im}\varphi=\operatorname{arg}g'$. 
However, this does not imply that $\varphi \in B_0(\mathbb{D})$, contrary to the claim made in \cite{GGPPR}. 

To see this, by \cite[Theorem 1]{Po78}, it is sufficient to show that the boundary curve is not asymptotically conformal.
Indeed, let $\Gamma$ denote the graph of the function $y = x^2 \cos(\pi/x^2)$ on $[-1,1]$. 
Set
\[
a_n=\tfrac1{\sqrt{2n-1/2}},\quad
w_n=\tfrac1{\sqrt{2n}}+i\tfrac1{2n},\quad
b_n=\tfrac1{\sqrt{2n+1/2}}
\]
for any positive integer $n \in \mathbb{N}$.
Then these points lie on $\Gamma$. 
A straightforward computation yields
\begin{align*}
   |a_n-w_n|+|w_n-b_n| \geq \tfrac{1}{n},\quad |a_n-b_n| \leq \tfrac{1}{2n^{3/2}}.
\end{align*}
Consequently,
\[
\frac{|a_n-w_n|+|w_n-b_n|}{|a_n-b_n|}\geq 2\sqrt{n} \to \infty \qquad (n\to\infty).
\]
Hence the curve determined by $y=x^2\cos\frac{\pi}{x^2}$ is not asymptotically conformal. 
In fact, it is not even a quasicircle.

Our construction is based on an iterative replacement procedure. 
Starting from a family of smooth graph arcs, we repeatedly subdivide each convex piece into sufficiently short subarcs and replace every subarc by a small oscillatory bump obtained by embedding a fixed graph into its local Frenet frame (see Figure~\ref{gamma}). 
The amplitude of each bump tends to zero, while the total increase of arc length at each generation is carefully controlled. 
As a consequence, the resulting curve remains asymptotically conformal and chord-arc, whereas the accumulated increase of arc length over infinitely many generations prevents the curve from being asymptotically smooth. 

Beyond the example for Theorem~\ref{main1}, we also establish a positive result: under an additional geometric condition, asymptotic conformality implies asymptotic smoothness for chord-arc curves.
A chord-arc curve $\Gamma$ is said to be \emph{uniformly approximable} if, for every $\varepsilon > 0$, there exists a positive integer $n \in \mathbb{N}$ such that every subarc $\wideparen{ab} \subset \Gamma$ with endpoints $a, b$ admits a partition $a = a_0, a_1, \ldots, a_n = b$ along $\wideparen{ab}$ satisfying
\[
(1+\varepsilon)\sum_{i=1}^n |a_i - a_{i-1}|\geq \ell(\wideparen{ab}).
\]
The gap between the two notions, asymptotic conformality and asymptotic smoothness, can thus be precisely bridged by this additional condition.
Namely, a bounded chord-arc curve is asymptotically smooth if and only if it is asymptotically conformal and uniformly approximable (Theorem~\ref{main2}).

As another positive result, we introduce an intermediate Teichm\"uller space $T_{B_0}$, identified with $\mathrm{SQS}(\mathbb{S}) \cap \mathrm{S}(\mathbb{S})$ modulo $\mathrm{M\ddot ob}(\mathbb{S})$, which lies between $T_B$ and $T_V$. 
We claim that $T_{B_0}$ shares the same basic properties as $T_B$ (Theorem~\ref{intermediate}).
More importantly, we investigate this space from the perspective of whether certain properties of $T_V$ can be extended to $T_{B_0}$. 
We raise questions in this direction and present the results obtained.

The organization of the paper is as follows.
Sections~\ref{counterexample} and \ref{verf} are devoted to the proof of Theorem~\ref{main1}, including preliminaries on curves, the construction of the example, and verification of the required properties, respectively.
In Section~\ref{charac}, we establish a characterization of asymptotic smoothness via uniform approximability (Theorem~\ref{main2}).
Finally, Section~\ref{Teich} surveys the $\mathrm{BMO}$ and $\mathrm{VMO}$ Teichm\"uller spaces and introduces an intermediate Teichm\"uller space between them.
In Theorem~\ref{intermediate}, we state its basic properties in quasiconformal Teichm\"uller theory, after which we discuss further questions and potential applications concerning this space.

\medskip
\noindent
{\bf Acknowledgements.}
The authors thank the reviewers for their extremely careful reading, checking formulas, and correcting mistakes.
Their valuable suggestions improved readability of this paper.


\section{Construction of the Example}\label{counterexample}
We say that a planar curve $\Gamma$ is of class $C^k$ if it admits a $k$-times continuously differentiable parametrization $\mathbf{r}(t)=\bigl(X(t),Y(t)\bigr)$ with $\dot{\mathbf{r}}(t)\neq 0$ for all $t$. 
Now, let $\Gamma$ be a $C^2$ planar curve, and let $\Gamma(s)$ denote its arc-length parametrization. 
We consider the natural moving frame along $\Gamma$ (see Figure~\ref{fs}). 
More precisely, there exists a local coordinate based on the Frenet frame given by
\[
    z(s, u)=\Gamma(s)+u N(s),
\]
where $T(s)=\dot{\Gamma}(s)$ is the unit tangent vector and $N(s)$ is the unit normal vector in the positive direction from $T(s)$. 
Here, $s$ represents the arc-length parameter (the position along the curve) and $u$ is the coordinate in the normal direction at $\Gamma(s)$.

\begin{figure}[htp]
 \centering
 \includegraphics[width=6cm]{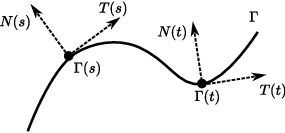}
 \caption{The Frenet frame field along $\Gamma$}
 \label{fs}
\end{figure}

The $2$-dimensional Frenet--Serret formulas describe the kinematic properties of a particle moving along a $C^2$ curve in the complex plane $\mathbb{C}$:
\[
    \dot{T}(s) = \kappa(s)N(s), \quad \dot{N}(s) = -\kappa(s)T(s).
\]

For each $h>0$, define
\[
    f_{h}(t)=h\sin^{2}(\pi t), \quad t\in[0,1].
\]
We also denote by $f_{h}$ the graph of the function $f_{h}(t)$. 
With the above notation, we embed the graph of $f_{h}(t)$ into the local coordinate system $(s,u)\mapsto z(s, u)$ for $ s\in[0,1]$ (see Figure~\ref{fsembed}). 
We say that $\Gamma(s)$ is the \emph{projection} of $\tilde{f}_h(s)$ into $\Gamma$ for $s\in[0,1]$, meaning that they share the same abscissa in the local coordinate system. 

Explicitly, the embedded curve $\tilde{f}_h(s)$ and its derivative are given by:
\begin{align*}
    \tilde{f}_h(s)&\eqqcolon z(s, f_h(s))=\Gamma(s)+f_h(s) N(s),\\
    \dot{\tilde{f}}_h(s)&=\dot{\Gamma}(s)+\dot{f}_h(s) N(s)+f_h(s) \dot{N}(s).
\end{align*}
\begin{figure}[htp]
 \centering
 \includegraphics[width=10cm]{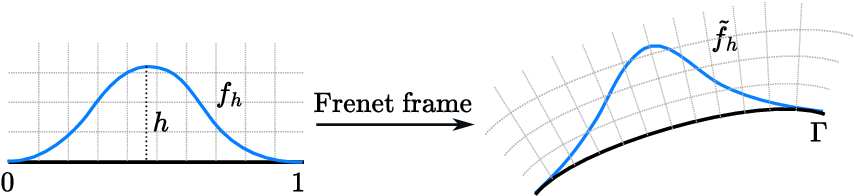}
 \caption{Illustration of $f_h$ embedded into the local Frenet frame of $\Gamma$}
 \label{fsembed}
\end{figure}
By the Frenet--Serret formulas, we obtain
\begin{align*}
    \dot{\tilde{f}}_h(s)
    &=T(s)+\dot{f}_{h} (s) N(s)-\kappa(s) f_{h}(s) T(s)\\
    &=\big(1-\kappa(s) f_{h}(s)\big) T(s)+\dot{f}_{h} (s) N(s).
\end{align*}
Therefore, the norm of $\dot{\tilde{f}}_h(s)$ is given by
\begin{align}\label{tangent}
    \big\|\dot{\tilde{f}}_h(s)\big\|
    &=\sqrt{\big(1-\kappa(s) f_h(s)\big)^2+\big(\dot{f}_{h}(s)\big)^2} \notag\\
    &=\sqrt{1+\big(\dot{f}_{h}(s)\big)^2-2 \kappa(s) f_h(s)+\kappa^2(s) f_h^2(s)}.
\end{align}

In this section, we shall construct a curve $\gamma$ that is asymptotically conformal and chord-arc, but not asymptotically smooth as required in Theorem~\ref{main1}. 

The construction proceeds as follows.

Denote $\gamma_{n}^{(1)}$ as the graph of 
\[
    f_{\frac{1}{n\cdot 2^{n}}}\left(2^{n}(t-\tfrac{1}{2^{n}})\right), \quad t\in\left[\tfrac{1}{2^{n}},\tfrac{1}{2^{n-1}}\right], \quad n=1,\,2,\cdots.
\]
See Figure~\ref{gamma1} below.

\begin{figure}[htp]
 \centering
 \includegraphics[width=8cm]{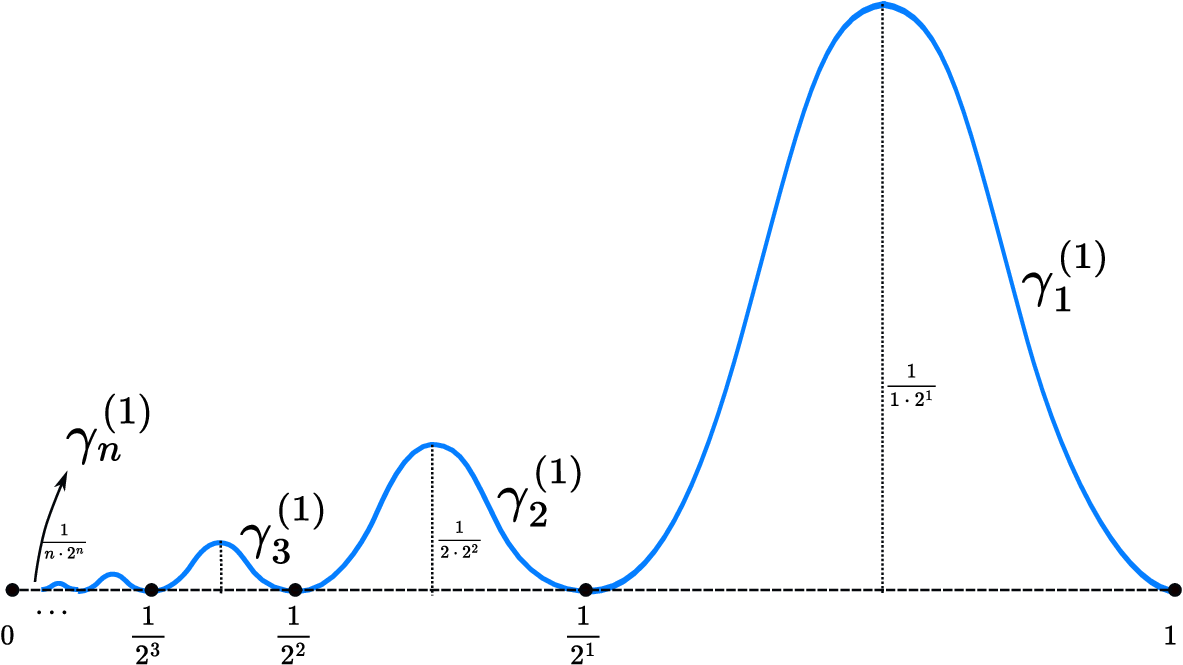}
 \caption{Graphs of $f_{\frac{1}{n\cdot 2^{n}}}\left(2^{n}\left(t-\frac{1}{2^{n}}\right)\right)$}
 \label{gamma1}
\end{figure}

Let $\Gamma(s)$ denote the arc-length parametrization of a planar curve $\Gamma$ and let $C\cdot\Gamma$ denote the curve obtained by scaling $\Gamma$ by a factor of $C>0$ while keeping $\Gamma(0)$ fixed. 
Set 
\[
    \beta_{n}^{(k)}=\frac{k}{n^{2}}, \quad k=1,\cdots, n.
\] 

For any fixed $n\geq 2$, we construct the curve $\gamma_n^{(2)}$ from $\gamma_n^{(1)}$ as follows (see Figure~\ref{gamma_2_from_gamma_1}). 
First, divide $\gamma_{n}^{(1)}$ into subarcs $\sigma_{i}^{(1)}$ determined by its inflection points. 
If necessary, reverse the orientation of $\sigma_{i}^{(1)}$ so that $\sigma_{i}^{(1)}$ can be considered as a strictly convex upward curve with negative curvature. 
Set 
\[
    \varepsilon_{n}^{(1)}=\min\left\{\tfrac{\sqrt{\beta_{n}^{(1)}}}{K_{n}^{(1)}},\frac{1}{2}\varepsilon_{n}^{(0)}\right\},\quad \varepsilon_{n}^{(0)}=\frac{1}{2^{n+1}},
\]
where the upper bound of the curvature $\kappa_{n}^{(1)}(s)$ of $\gamma_{n}^{(1)}$ is given as
\[
    K_{n}^{(1)}= \sup_{s}|\kappa_{n}^{(1)}(s)|.
\]

For each arc $\sigma_{i}^{(1)}$, there exist a constant $\alpha_{i}^{(1)} \in [1,2)$ and an integer $N_i^{(1)} \in \mathbb{N}$ such that $\sigma_{i}^{(1)}$ can be partitioned into subarcs $\{\sigma_{i_m}^{(1)}\}_{m=1}^{N_i^{(1)}}$ satisfying
\[
    \ell\bigl(\sigma_{i_m}^{(1)}\bigr) = \alpha_{i}^{(1)}\varepsilon_{n}^{(1)}, \quad m = 1,\ldots,N_i^{(1)}.
\]
Indeed, the integer $N_i^{(1)}$ is uniquely determined by the inequalities
\[
    N_i^{(1)}\varepsilon_{n}^{(1)} \leq \ell\big(\sigma_i^{(1)}\big) < \big(N_i^{(1)}+1\big)\varepsilon_{n}^{(1)},
\]
and then we set
\[
    \alpha_i^{(1)} =\frac{\ell\big(\sigma_i^{(1)}\big)}{N_i^{(1)}\varepsilon_{n}^{(1)}}.
\]

Consider the curve obtained by embedding $f_{\sqrt{\beta_n^{(1)}}}$ into the local coordinate of the normalized arc $\ell\big(\sigma_{i_m}^{(1)}\big)^{-1}\cdot\sigma_{i_{m}}^{(1)}$, and denote this curve by $\Lambda_{i_{m}}^{(1)}$. 
Rescaling $\Lambda_{i_{m}}^{(1)}$ back by the factor $\ell\big(\sigma_{i_m}^{(1)}\big)$ yields a rescaled curve, denoted by $\lambda_{i_ m}^{(1)}$; namely, 
\[
    \lambda_{i_ m}^{(1)}=\ell\big(\sigma_{i_m}^{(1)}\big)\cdot\Lambda_{i_{m}}^{(1)}.
\]
We then define
\[
    \gamma_n^{(2)}\coloneqq\bigcup_{i,\, m}\lambda_{i_m}^{(1)}.
\]
\begin{figure}[htp]
 \centering
 \includegraphics[width=10cm]{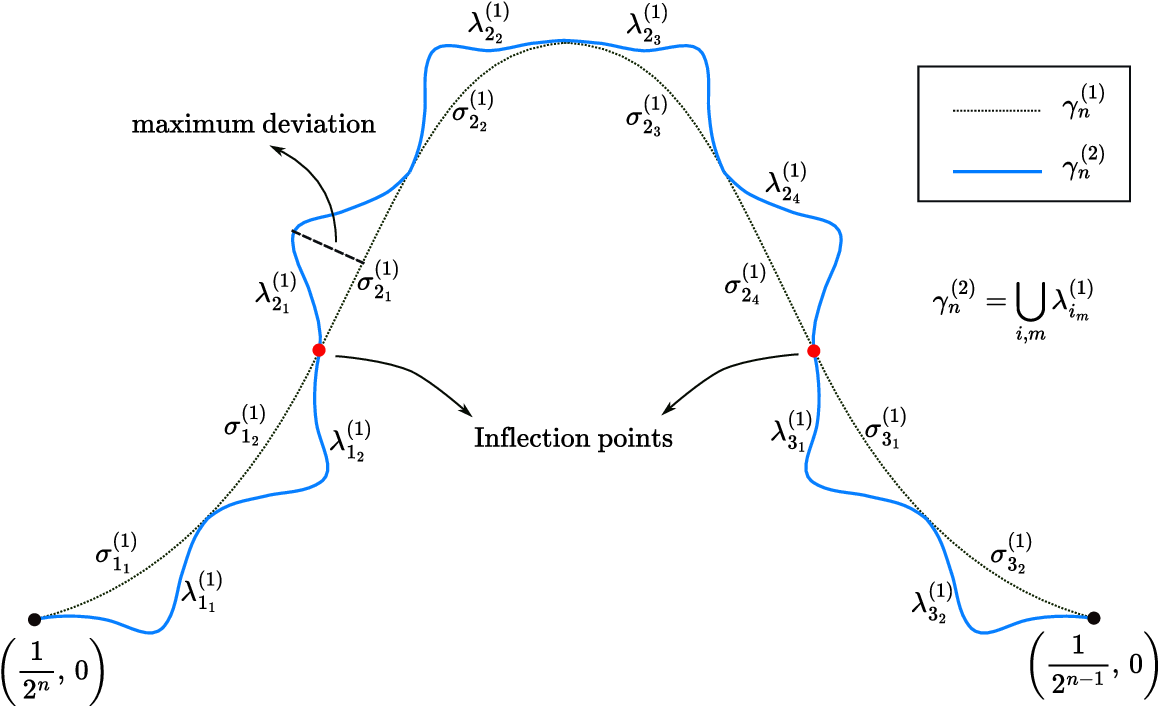}
 \caption{Illustration of how $\gamma_n^{(2)}$ is obtained from $\gamma_n^{(1)}$}
 \label{gamma_2_from_gamma_1}
\end{figure}

For any $k\leq n$ with a fixed $n$, $\gamma_{n}^{(k)}$ is constructed inductively in the following way. 
Divide $\gamma_{n}^{(k-1)}$ into subarcs $\sigma_{i}^{(k-1)}$ by its inflection points. 
If necessary, reverse the orientation of $\sigma_{i}^{(k-1)}$ so that each subarc $\sigma_{i}^{(k-1)}$ can be regarded as a strictly convex upward curve with negative curvature. 
Set 
\[
    \varepsilon_{n}^{(k-1)}=\min\left\{\tfrac{\sqrt{\beta_{n}^{(k-1)}}}{K_{n}^{(k-1)}},\frac{1}{2}\varepsilon_{n}^{(k-2)}\right\},
\]
where the upper bound of the curvature $\kappa_{n}^{(k-1)}(s)$ of $\gamma_{n}^{(k-1)}$ is given as
\[
    K_{n}^{(k-1)}=\sup_s|\kappa_{n}^{(k-1)}(s)|.
\]

The choice of $\varepsilon_{n}^{(k-1)}$ serves two different purposes. 
The first condition ensures that the original curve is almost flat on each selected subarc. 
Therefore, every inserted bump is a small perturbation of a relatively flat segment, making the geometry quantitatively controllable. 
The second condition guarantees that the perturbation scale decreases geometrically, so each new generation acts only on a much finer scale than the previous one.

By appropriately choosing $\alpha_{i}^{(k-1)}\in[1,2)$ and $N_i^{(k-1)} \in \mathbb{N}$ as in the previous step, we can partition $\sigma_{i}^{(k-1)}$ into subarcs $\sigma_{i_{m}}^{(k-1)}$ such that
\begin{align}\label{length_sigma_im_k-1}
     \ell\big(\sigma_{i_m}^{(k-1)}\big)=\alpha_{i}^{(k-1)}\varepsilon_{n}^{(k-1)}, \quad m=1,\cdots, N_i^{(k-1)}.
\end{align}
For each normalized arc $\ell\big(\sigma_{i_m}^{(k-1)}\big)^{-1}\cdot\sigma_{i_m}^{(k-1)}$, embed $f_{\sqrt{\beta_n^{(k-1)}}}$ into its local coordinate, and denote the resulting curve by $\Lambda_{i_{m}}^{(k-1)}$. 
Rescaling $\Lambda_{i_{m}}^{(k-1)}$ back by the factor $\ell\big(\sigma_{i_m}^{(k-1)}\big)$ yields the curve
\begin{equation}\label{lambda_im_k-1}
    \lambda_{i_ m}^{(k-1)}=\ell\big(\sigma_{i_m}^{(k-1)}\big)\cdot\Lambda_{i_{m}}^{(k-1)}.
\end{equation}
Then define
\[
    \gamma_n^{(k)}\coloneqq\bigcup_{i,\,m}\lambda_{i_m}^{(k-1)}.
\]

The parameter $\beta_n^{(k)} = \frac{k}{n^2}$ determines the height of the newly inserted bump. 
Since the relative increase of arc length produced by one insertion is approximately $1 + \beta_n^{(k)}$, after $n$ generations the total length satisfies
\[
\prod_{k=1}^n \left(1 + \beta_n^{(k)}\right) \approx \exp\left(\sum_{k=1}^n \beta_n^{(k)}\right).
\]
Our choice gives
\[
\sum_{k=1}^n \beta_n^{(k)} = \frac{1}{n^2} \sum_{k=1}^n k \sim \frac{1}{2},
\]
so the accumulated increase of length remains bounded away from $1$. 
This guarantees that the limiting curve fails to be asymptotically smooth.

The curve $\gamma$ for Theorem~\ref{main1} is defined by
\begin{align*}
    \gamma\coloneqq\left( \bigcup_{n=1}^{\infty} \gamma_{n}^{(n)} \right) \ \cup \ C_{1} \ \cup \ C_{2} \ \cup \ C_{3},
\end{align*}
where the components are defined by
\begin{align*}
    & C_{1}\coloneqq\left\{z:\left|z+\tfrac{i}{8}\right|=\tfrac{1}{8},\operatorname{Re} z\leq 0\right\},\\
    & C_{2}\coloneqq\left\{z:\operatorname{Im} z=-\tfrac{1}{4}, 0\leq\operatorname{Re} z\leq 1\right\},\\
    & C_{3}\coloneqq\left\{z:\left|z-\left(1-\tfrac{i}{8}\right)\right|=\tfrac{1}{8},\operatorname{Re} z\geq 1\right\}.
\end{align*}
See Figure~\ref{gamma} below.

\begin{figure}[htp]
 \centering
 \includegraphics[width=10cm]{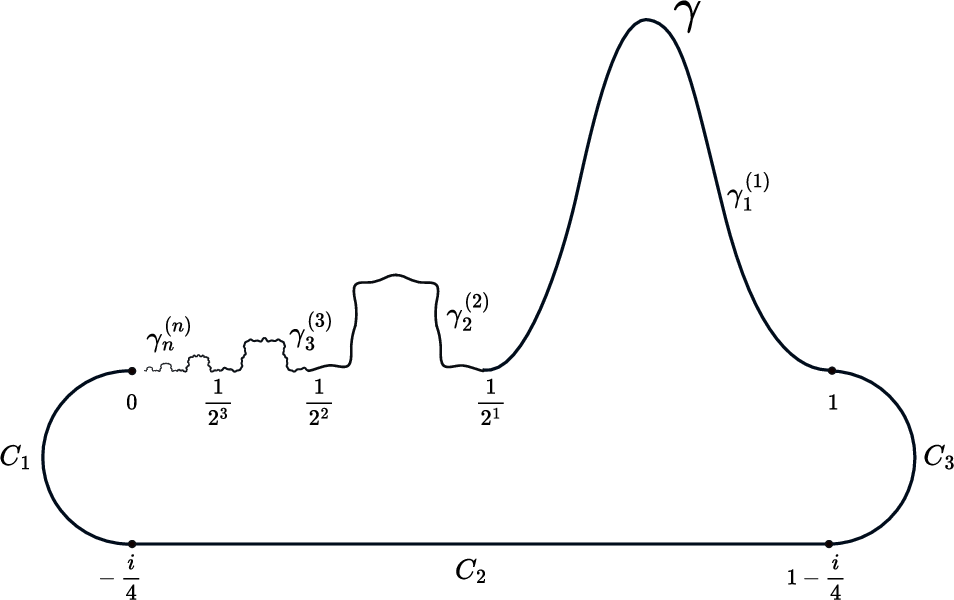}
 \caption{Construction of the curve $\gamma$}
 \label{gamma}
\end{figure}

In summary, the parameters $\beta_n^{(k)}$ and $\varepsilon_n^{(k)}$ play complementary roles. 
The former controls the accumulated increase of arc length, which eventually destroys asymptotic smoothness, whereas the latter localizes every perturbation to an increasingly finer scale, ensuring that the overall geometry remains asymptotically conformal and chord-arc.

\section{Verification of the Example Curve}\label{verf}
In this section, we prove that the curve constructed in the previous section serves as an example that an asymptotically conformal chord-arc curve is not necessarily asymptotically smooth. 
This implies Theorem~\ref{main1}. In the first subsection, we prepare lemmas and claims for the proof in the second subsection.

\subsection{Lemmas and Claims}

First, we state basic estimates as lemmas.

\begin{lemma}\label{lengthest_of_embeded_curve}
    Let $\Gamma$ be a $C^2$ planar curve of length $1$, parametrized by arc length $\Gamma(s)$ for $s \in [0,1]$. 
    Moreover, suppose that its curvature $\kappa(s)$ satisfies
    \[
        -1 < \kappa(s) < 0, \quad  s \in [0,1].
    \]  
    Let $\tilde{f}_h$ be the curve obtained by embedding $f_h$ into the local coordinates system of $\Gamma$. 
    Then, for all sufficiently small $h > 0$, the length $\ell\big(\tilde{f}_h\big)$ satisfies 
    \[
        1 + h^2 \leq \ell\big(\tilde{f}_h\big) \leq 1 + 4h^2 + Kh,
    \]
    where $K = \sup_{s \in [0,1]}|\kappa(s)|$.
\end{lemma}
\begin{proof}
By applying a second-order Taylor expansion $\sqrt{1+x} = 1 + \frac{1}{2}x - \frac{1}{8}x^2 + O(x^3)$ to the formula $\big\|\dot{\tilde f}_{h}(s) \big\|$ 
in \eqref{tangent} around $h=0$, and by using the Frenet coordinate expression for $\tilde{f}_h$, we obtain
\begin{align*}
    \ell\big(\tilde{f}_{h}\big)=
     &\int_{0}^{1}\big \|\dot{\tilde f}_{h} (s)\big\|d s\\
    =&\int_0^1\sqrt{1+\big(\dot{f}_{h}(s)\big)^2-2\,\kappa(s) f_h(s)+\kappa^2(s) f_h^2(s)}ds\\
    =&1+\int_0^1\left(\tfrac{1}{2}\big(\dot{f}_{h}(s)\big)^2-\kappa(s) f_{h}(s)+\tfrac{1}{2} \kappa^2(s) f_h^2(s)-\tfrac{1}{2}\kappa^2(s)f_h^2(s)\right) d s+O(h^3)\\
    =&1+\int_0^1\left(\tfrac{1}{2}\big(\dot{f}_{h}(s)\big)^2-\kappa(s) f_{h}(s)\right) d s+O(h^3),
\end{align*}
where the $O(h^3)$ term follows from the fact that $f_h$ and $\dot{f}_h$ are of order $O(h)$.

Since $-\kappa(s) f_h(s)$ is nonnegative by the assumption on $\kappa$, we may drop this term to obtain
\begin{align*}
    \ell\big(\tilde{f}_{h}\big)\geq 1+\frac{1}{2}\int_{0}^{1}\big(\dot{f}_{h} (s)\big)^{2} d s+O(h^3)=1+\frac{\pi^{2}}{4} h^{2}+O(h^3)\geq 1+h^2
\end{align*}
for all sufficiently small $h$.

On the other hand, using the uniform curvature bound $K$, we have
\begin{align*}
    \ell\big(\tilde{f}_h\big)
    &\leq 1+\frac{1}{2}\int_0^1\big(\dot{f}_{h}(s)\big)^2ds+K\int_0^1 f_h(s) ds+O(h^3)\\
    &=1+\tfrac{\pi^2}{4} h^2+\tfrac{K}{2} h+O(h^3)\\
    &\leq 1+4 h^2+K h
\end{align*}
for sufficiently small $h$. 
This yields the desired two-sided estimate.
\end{proof}

We define the \emph{maximum deviation} of a curve $\Gamma_1$ from another curve $\Gamma_2$ as
\[
    D\left(\Gamma_1,\Gamma_2\right)\coloneqq \sup_{x\in\Gamma_1}\ \operatorname{dist}\left(x,\Gamma_2\right),
\]
where $\operatorname{dist}(y, \Gamma)\coloneqq  \inf_{z \in \Gamma} |y - z|$ is the Euclidean distance from $y$ to $\Gamma$. 

\begin{lemma}\label{dist_from_fh_to _gamma}
    Let $\Gamma$ be a $C^2$ planar curve of length $1$, parametrized by arc length $\Gamma(s)$ for $s \in [0,1]$. 
    Let $\tilde{f}_h$ be the curve obtained by embedding $f_h$ into the local coordinates system of $\Gamma$. 
    Then, the maximum deviation of $\tilde{f}_{h}$ from $\Gamma$ satisfies 
    \[
        D\big(\tilde{f}_{h}, \Gamma\big ) \leq h
    \]
    for all $h > 0$.
\end{lemma}
\begin{proof}
For any $t \in [0,1]$, the distance from $\tilde{f}_h(t)$ to $\Gamma$ is  
\[
    \operatorname{dist}\big(\tilde{f}_h(t), \Gamma\big) = \inf_{s \in [0,1]} \big|\tilde{f}_h(t) - \Gamma(s) \big|\leq \big|\tilde{f}_h(t) - \Gamma(t) \big|=|f_h(t)|\leq h.  
\]
Taking the supremum over $t \in [0,1]$ yields the statement.
\end{proof}

This lemma yields the following claim.

\begin{claim}\label{D_gammak_to_gamman}
    The maximum deviation of $\gamma_n^{(n)}$ from $\gamma_n^{(k)}$ satisfies
    \[
        D\big(\gamma_n^{(n)},\gamma_n^{(k)}\big)\leq \frac{4\varepsilon_{n}^{(k)}}{\sqrt{n}}.
    \]
\end{claim}
\begin{proof}
   By Lemma~\ref{dist_from_fh_to _gamma} and the definition of the maximum deviation, we obtain
\begin{align}\label{eq: gnk}
    D\big(\gamma_{n}^{(k)},\gamma_{n}^{(k-1)}\big)
    &=\sup_{i,\, m}\ D\big(\lambda_{i_m}^{(k-1)},\sigma_{i_m}^{(k-1)}\big)\notag\\
    &\leq\sup_{i,\, m} \ \ell\big(\sigma_{i_m}^{(k-1)}\big)D\Big(\Lambda_{i_m}^{(k-1)},\ell\big(\sigma_{i_m}^{(k-1)}\big)^{-1}\cdot\sigma_{i_m}^{(k-1)}\Big)\notag\\
    &\leq\sup_{i\, m}\ell\big(\sigma_{i_m}^{(k-1)}\big)\sqrt{\beta_n^{(k-1)}}\leq 2\varepsilon_n^{(k-1)}\sqrt{\beta_n^{(k-1)}},
 \end{align}
where the last inequality follows from the equality \eqref{length_sigma_im_k-1}.

From the definition of $\varepsilon_{n}^{(j)}$, we have
\[
    \varepsilon_{n}^{(j)}\leq\frac{1}{2^{j-k}}\varepsilon_{n}^{(k)},\quad j\geq k+1.
\]
The above estimate, together with inequality \eqref{eq: gnk}, implies that
\begin{align*}
    D\big(\gamma_n^{(n)},\gamma_n^{(k)}\big)
    &\leq\sum_{j=k+1}^n D\big(\gamma_n^{(j)},\gamma_n^{(j-1)}\big)\leq 2\sum_{j=k+1}^n\varepsilon_n^{(j-1)}\sqrt{\beta_n^{(j-1)}}\\
    &=2\sum_{j=k}^{n-1}\varepsilon_n^{(j)}\sqrt{\beta_n^{(j)}}\leq 2\sum_{j=k}^{n-1}\frac{1}{2^{j-k}}\sqrt{\beta_n^{(j)}}\varepsilon_n^{(k)}\leq\frac{4\varepsilon_n^{(k)}}{\sqrt{n}},
\end{align*}
where in the last inequality we use the bound $\beta_n^{(j)}\leq \frac{1}{n}$.
\end{proof}

We prepare three more claims which are used in the later argumnets.

\begin{claim}\label{projection}
   Let $x_{1}^{(k)},\,x_{2}^{(k)}$ be two points on the curve $\lambda_{i_m}^{(k)}$, a bump of the curve $\gamma_n^{(k+1)}$ defined in \eqref{lambda_im_k-1}, and let $x_{1}^{(k-1)},\,x_{2}^{(k-1)}$ be their projections on the curve $\sigma_{i_m}^{(k)}$ for $1 \leq k \leq n-1$. 
   Then, for sufficiently large $n$, the ratio of arc lengths satisfies
    \[
    \frac{\ell\big(\wideparen{{x_{1}^{(k)}x_{2}^{(k)}}}\big)}{\ell\big(\wideparen{x_{1}^{(k-1)}x_{2}^{(k-1)}}\big)}\leq 1+8\beta_{n}^{(k)}.
    \]
\end{claim}
\begin{proof}
Consider the magnified curves $\ell\big(\sigma_{i_m}^{(k)}\big)^{-1} \cdot \sigma_{i_m}^{(k)}$ and $\ell\big(\sigma_{i_m}^{(k)}\big)^{-1} \cdot \lambda_{i_m}^{(k)}$, and let $\hat{x}_{i}^{(k)}$ and $\hat{x}_{i}^{(k-1)}$ be the points corresponding to $x_{i}^{(k)}$ and $x_{i}^{(k-1)}$ on these magnified curves, respectively, $i = 1, 2$. 
Suppose that $\Gamma(s)$ is the arc-length parametrization of $\ell\big(\sigma_{i_m}^{(k)}\big)^{-1} \cdot \sigma_{i_m}^{(k)}$, $s\in[0,1]$ and set
\[
    \hat{x}_{1}^{(k-1)} = \Gamma(t), \quad \hat{x}_{2}^{(k-1)} = \Gamma(r),\quad 0\leq t\leq r\leq 1.
\]
Then the corresponding points on the magnified curve are given by
\[
    \begin{cases}
        \hat{x}_{1}^{(k)} = \Gamma(t) + f_{\sqrt{\beta_n^{(k)}}}(t) \, N(t), \\
        \hat{x}_{2}^{(k)} = \Gamma(r) + f_{\sqrt{\beta_n^{(k)}}}(r) \, N(r).
    \end{cases}
\]
By the definition of $\varepsilon_{n}^{(k)}$, the curvature of the normalized curve $\ell\big(\sigma_{i_m}^{(k)}\big)^{-1} \cdot \sigma_{i_m}^{(k)}$ parametrized by arc length satisfies 
\begin{equation}\label{curvature}
    |\kappa(s)|=\ell\big(\sigma_{i_m}^{(k)}\big)|\kappa_n^{(k)}(s)|
    \leq \ell\big(\sigma_{i_m}^{(k)}\big) \frac{\sqrt{\beta_{n}^{(k)}}}{\varepsilon_n^{(k)}} \leq 2\sqrt{\beta_{n}^{(k)}}, \quad s\in[0,1].
\end{equation}

By representing $\ell\big(\wideparen{\hat{x}_{1}^{(k)}\hat{x}_{2}^{(k)}}\,\,\,\,\big)$ as in \eqref{tangent} and noting that 
\begin{align}\label{eq: f'f_est}
    \lvert \dot{f}_{\sqrt{\beta_{n}^{(k)}}}(u)\rvert \leq \pi\sqrt{\beta_{n}^{(k)}}, \quad \lvert f_{\sqrt{\beta_{n}^{(k)}}}(u)\rvert \leq \sqrt{\beta_{n}^{(k)}},
\end{align}
a second-order Taylor expansion $\sqrt{1+x} = 1 + \frac{1}{2}x - \frac{1}{8}x^2 + O(x^3)$ with $x = O(\beta_n^{(k)})$ yields
\begin{align*}
    &\ell\big(\wideparen{\hat{x}_{1}^{(k)}\hat{x}_{2}^{(k)}}\,\,\,\,\big)\\
    =&\int_{t}^{r}\sqrt{1+\big(\dot{f}_{\sqrt{\beta_{n}^{(k)}}}(u)\big)^{2}-2\kappa(u)f_{\sqrt{\beta_{n}^{(k)}}}(u)+\kappa^{2}(u)f_{\sqrt{\beta_{n}^{(k)}}}^{2}(u)}\,du\\
    =&\int_{t}^{r}\left(1+\frac{1}{2}\big(\dot{f}_{\sqrt{\beta_{n}^{(k)}}}(u)\big)^{2}-\kappa(u)f_{\sqrt{\beta_{n}^{(k)}}}(u)+O\bigl((\beta_{n}^{(k)})^2\bigr)\right)\,du\\
    \leq&(r-t)+\frac{\pi^{2}}{2}\beta_{n}^{(k)}(r-t)+2\beta_{n}^{(k)}(r-t)+O\bigl((\beta_{n}^{(k)})^2\bigr)(r-t)\\
    \leq&(r-t)\big(1+8\beta_{n}^{(k)}\big)
\end{align*}
for sufficiently large $n$. 
We conclude that
\[
    \frac{\ell\big(\wideparen{x_{1}^{(k)}x_{2}^{(k)}}\big)}{\ell\big(\wideparen{x_{1}^{(k-1)}x_{2}^{(k-1)}}\big)}=\frac{\ell\big(\wideparen{\hat{x}_{1}^{(k)}\hat{x}_{2}^{(k)}}\,\,\,\,\big)}{\ell\big(\wideparen{\hat{x}_{1}^{(k-1)}\hat{x}_{2}^{(k-1)}}\,\,\,\,\big)}\leq\frac{(r-t)\big(1+8\beta_{n}^{(k)}\big)}{r-t}=1+8\beta_{n}^{(k)}.
\]
This completes the proof.
\end{proof}
\begin{claim}\label{asymconf_gamma1}
    Let $n_{1},\, n_{2} \in \mathbb{N}$, and let $x_{1} \in \gamma_{n_{1}}^{(1)}$, $x_{2} \in \gamma_{n_{2}}^{(1)}$.
    For any point $w$ on the arc $\wideparen{x_{1}x_{2}}$ connecting $x_{1}$ and $x_{2}$, we have
    \[
        \max_{w \in \wideparen{x_{1}x_{2}}}
        \frac{|x_{1} - w|+ |w - x_{2}|}{|x_{1} - x_{2}|}
        \leq 1 + \frac{\pi}{\min\{n_{1}, n_{2}\}}.
    \]
\end{claim}
\begin{proof}
Without loss of generality, assume $n_{1} \geq n_{2}$, and set
\begin{align*}
     x_{1}&=\left(t, f_{\frac{1}{n_{1}\cdot 2^{n_{1}}}}\left(2^{n_{1}}(t-\tfrac{1}{2^{n_{1}}})\right)\right),\\
     x_{2}&=\left(r, f_{\frac{1}{n_{2}\cdot 2^{n_{2}}}}\left(2^{n_{2}}(r-\tfrac{1}{2^{n_{2}}})\right)\right).
\end{align*}
Observe that 
\begin{align*}
    \sqrt{1+\left(\dot{f}_{\frac{1}{n\cdot 2^{n}}}\left(2^{n}(t-\tfrac{1}{2^{n}})\right)\right)^2}\leq 1+\left|\dot{f}_{\frac{1}{n\cdot 2^{n}}}\left(2^{n}(t-\tfrac{1}{2^{n}})\right)\right|\leq 1+\frac {\pi}{n},
\end{align*}
which implies
\begin{align}\label{n1=n2}
    |x_{1}-w|+ |w-x_{2}|\leq \ell(\wideparen{x_{1}x_{2}}) \leq \big(1+\frac{\pi}{n_{2}}\big)(r-t).
\end{align}

Since $|x_{1}-x_{2}|\geq |r-t|$, combining with \eqref{n1=n2} yields
\[
    \frac{|x_{1}-w|+ |w-x_{2}|}{|x_{1}-x_{2}|} \leq 1 + \frac{\pi}{n_{2}}.
\]
This yields the statement.
\end{proof}

\begin{claim}\label{slope}
    Let $\Gamma(s) = \bigl(X(s), Y(s)\bigr)$, $s \in [0,1]$, be a $C^2$ convex upward curve parametrized by arc length, with $\Gamma(0) = (0,0)$ and the curvature satisfying $|\kappa(s)|\leq 2\sqrt{\beta_{n}^{(k)}}$. 
    Assume further that
    \begin{enumerate}[\rm(i\rm)]
        \item $X'(s) > 0$ for all $s \in (0,1)$ {\rm;}  
        \item the derivative $(Y \circ X^{-1})'(t) \to 0$ as $t = X(s) \to 0$.   
    \end{enumerate}
    Consider the perturbed curve
    \[
        \Lambda(s) = \Gamma(s) + f_{\sqrt{\beta_{n}^{(k)}}}(s) N(s),
    \]  
    where $N(s)$ is the unit normal vector to $\Gamma$ at $s$. 
    Then, for sufficiently large $n$, $\Lambda$ can be reparametrized as a graph over the $X$-axis
    \[
        \Lambda(t) = \bigl(t, \tilde{Y}(t)\bigr)
    \]  
    such that
    \[
        \tilde{Y}(t) \leq 8\sqrt{\beta_{n}^{(k)}}t.
    \]
\end{claim}
\begin{proof}
The Frenet frame for the curve $\Gamma$ consists of the tangent and normal vectors
\[
    T(s)=\bigl(X'(s),Y'(s)\bigr), \quad N(s)=\bigl(-Y'(s),X'(s)\bigr).
\]
The curvature of a planar curve is defined as the rate of change of the tangent angle $ \theta(s)$ with respect to arc length parameter $s$, that is, $\kappa = d\theta/ds$. 

Condition (ii) ensures that
\[
    \frac{Y'(s)}{X'(s)}=(Y\circ X^{-1})'(t)\to 0 \quad \text{as } s\to 0,
\]
so that $\theta(s)\to 0$ as $s\to 0$. Hence
\[
    |\theta(s)|=|\theta(s)-\theta(0)|\leq \int_{0}^{s}|\kappa(t)|\,dt 
    \leq 2\sqrt{\beta_{n}^{(k)}}\,s \leq 2\sqrt{\beta_{n}^{(k)}}.
\]
For sufficiently large $n$, this bound for $\theta$ yields
\begin{align}\label{x'}
    \frac{4}{5}< X'(s)&=\cos\theta(s)\leq 1,\notag\\
    |Y'(s)|&=|\sin\theta(s)|\leq|\theta(s)|\leq 2\sqrt{\beta_{n}^{(k)}}.
\end{align}
Differentiating again and applying the Frenet--Serret formulas give
\[  
    T'(s)=\bigl(X''(s),Y''(s)\bigr)=\kappa(s)N(s)=\bigl(-Y'(s)\kappa(s),X'(s)\kappa(s)\bigr),
\]
and thus by \eqref{x'} and the assumption $|\kappa(s)|\leq 2\sqrt{\beta_{n}^{(k)}}$, we obtain
\begin{align}\label{x''}
    |X''(s)|&=|Y'(s)||\kappa(s)|\leq 4\beta_{n}^{(k)},\notag\\
    |Y''(s)|&=|X'(s)||\kappa(s)|\leq 2\sqrt{\beta_{n}^{(k)}}.
\end{align}

Now consider the perturbed curve
\begin{align*}
    \Lambda(s)&=\Gamma(s)+f_{\sqrt{\beta_{n}^{(k)}}}(s)N(s)\eqqcolon\bigl(\hat{X}(s),\hat{Y}(s)\bigr),
\end{align*}
where 
\[
    \hat{X}(s)=X(s)-Y'(s)f_{\sqrt{\beta_{n}^{(k)}}}(s), \quad \hat{Y}(s)=Y(s)+X'(s)f_{\sqrt{\beta_{n}^{(k)}}}(s).
\]
Differentiating $\hat{X}$ with respect to $s$, we have
\begin{align}\label{X'}
    \hat{X}'(s)&=X'(s)-Y''(s)f_{\sqrt{\beta_{n}^{(k)}}}(s)-Y'(s)\dot{f}_{\sqrt{\beta_{n}^{(k)}}}(s)\notag\\
    &\geq\frac{4}{5}-2\beta_{n}^{(k)}-2\pi\beta_{n}^{(k)}>\frac{3}{4}.
\end{align}
Thus $\hat{X}'(s)>0$, and $\Lambda$ can indeed be reparametrized by $t=\hat{X}(s)$ as
\[
    \Lambda=\left(t,\hat{Y}\big(\hat{X}^{-1}(t)\big)\right)\eqqcolon\big(t,\tilde{Y}(t)\big).
\]
Applying \eqref{X'}, it follows that
\begin{align}\label{YX^-1}
    \big|\left(\hat{Y}\bigl(\hat{X}^{-1}(t)\bigr)\right)'\big |=\bigl|\hat{Y}'\bigl(\hat{X}^{-1}(t)\bigr)\cdot\bigl(\hat{X}^{-1}(t)\bigr)'\bigr|= \left|\frac{\hat{Y}'(s)}{\hat{X}'(s)}\right|\leq\frac{4}{3}|\hat{Y}'(s)|.
\end{align}

Next, estimates \eqref{x'} and \eqref{x''} yield
\begin{align}\label{Y'}
    |\hat{Y}'(s)|&= |Y'(s)+X''(s)f_{\sqrt{\beta_{n}^{(k)}}}(s)+X'(s)\dot{f}_{\sqrt{\beta_{n}^{(k)}}}(s)|\notag\\
    &\leq |Y'(s)|+|X''(s)f_{\sqrt{\beta_{n}^{(k)}}}(s)|+|X'(s)\dot{f}_{\sqrt{\beta_{n}^{(k)}}}(s)|\notag\\
    &\leq 2\sqrt{\beta_{n}^{(k)}}+4\beta_{n}^{(k)}\sqrt{\beta_{n}^{(k)}}+\pi\sqrt{\beta_{n}^{(k)}}\leq 6\sqrt{\beta_{n}^{(k)}}.
\end{align}
Combining \eqref{YX^-1} with \eqref{Y'}, we conclude
\[
    |\tilde{Y}'(t)|=\bigl|\left(\hat{Y}\bigl(\hat{X}^{-1}(t)\bigr)\right)'\bigr|\leq 8\sqrt{\beta_{n}^{(k)}}.
\]
Since $\tilde{Y}(0) = 0$, integration gives
\begin{align*}
    \tilde{Y}(t)\leq 8\sqrt{\beta_{n}^{(k)}}t.
\end{align*}
This completes the proof.
\end{proof}

Now we start on estimating the length of an arc and the distance between two points on our curve.
Recall that 
\[
    \gamma_{n}^{(k)}=\bigcup_{i,\, m}\lambda_{i_m}^{(k-1)},\quad f_{h}(t)=h\cdot\sin^{2}(\pi t),\quad \beta_{n}^{(k)}=\frac{k}{n^{2}}.
\]
We call $\lambda_{i_m}^{(k-1)}$ a \emph{bump} of the curve $\gamma_{n}^{(k)}$. 
We say $x_{1},\,x_{2} $ lie on \emph{adjacent bumps} of $\gamma_{n}^{(k)}$ if 
\[
    x_{1},\,x_{2} \in \lambda_{i_{m-1}}^{(k-1)} \cup \lambda_{i_{m}}^{(k-1)}\quad \text{or} \quad  x_{1},\,x_{2} \in \lambda_{(i-1)_M}^{(k-1)} \cup \lambda_{i_{1}}^{(k-1)},
\]
   where 
\[
    \sigma_{i_{m-1}}^{(k-1)} \cup \sigma_{i_{m}}^{(k-1)} \quad \text{or} \quad \sigma_{(i-1)_M}^{(k-1)} \cup \sigma_{i_{1}}^{(k-1)}
\]
are adjacent subarcs in the partition of $\gamma_{n}^{(k-1)}$, $\sigma_{(i-1)_M}^{(k-1)}$ is the last subarc of $\sigma_{i-1}^{(k-1)}$, and $\sigma_{i_{1}}^{(k-1)}$ is the first subarc of $\sigma_i^{(k-1)}$ with $M=N_{i-1}^{(k-1)}$ and $2\leq k \leq n$.
\begin{lemma}\label{x1x2_in_small_bump}
    For any sufficiently large $n$, let $x_1, \, x_2$ lie on adjacent bumps of $\gamma_{n}^{(k+1)}$ with $2\leq k\leq n$. Then
    \[
        \frac{\ell\big(\wideparen{x_{1}x_{2}}\big)}{|x_{1}-x_{2}|} \leq 1 + \frac{32}{\sqrt{n}}.
    \]
    In particular, the same estimate holds if $x_1$ and $x_2$ lie on the same bump of $\gamma_n^{(k+1)}$.
\end{lemma}
\begin{proof}
First, assume $x_{1},\,x_{2} \in \lambda_{i_{m-1}}^{(k)} \cup \lambda_{i_{m}}^{(k)}$, and let $\tilde{x}_{1},\tilde{x}_{2}$ be points on $\Lambda_{i_{m-1}}^{(k)}\cup\Lambda_{i_{m}}^{(k)}$ corresponding to $x_{1},\,x_{2}$, respectively (see Figure~\ref{l1l2}). 
Assume $\Gamma(s)$ is the arc-length parametrization of the curve $\tfrac{1}{\alpha_{i}^{(k)}\varepsilon_{n}^{(k)}}\cdot\big(\sigma_{i_{m-1}}^{(k)}\cup\sigma_{i_{m}}^{(k)}\big)$ for $s\in[0,2]$. Set
\[
    \begin{cases}
        \tilde{x}_{1}=\Gamma(t)+f_{\sqrt{\beta_{n}^{(k)}}}(t)N(t),\\
        \tilde{x}_{2}=\Gamma(r)+f_{\sqrt{\beta_{n}^{(k)}}}(r)N(r) 
    \end{cases}
\]
for $0 \leq t \leq 1 \leq r \leq 2$.

\begin{figure}[htp]
 \centering
 \includegraphics[width=8cm]{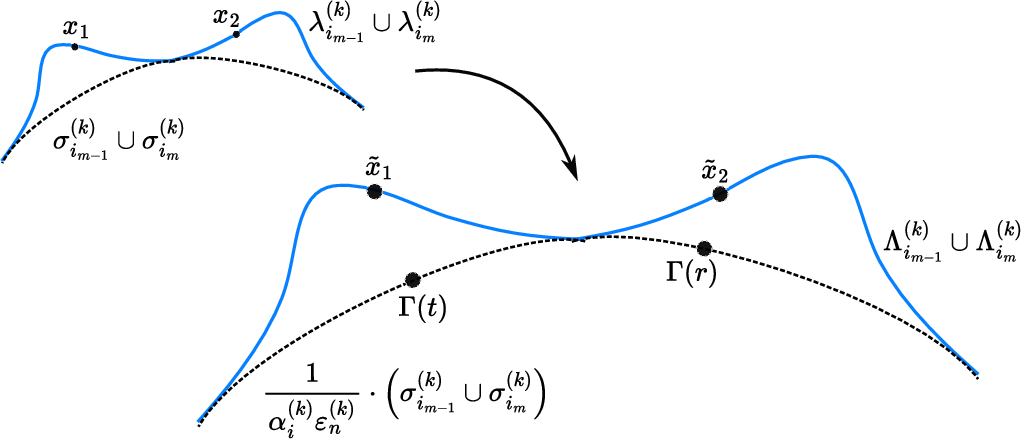}
 \caption{Adjacent subarcs $\sigma_{i_{m-1}}^{(k)} \cup \sigma_{i_{m}}^{(k)}$ and $ \lambda_{i_{m-1}}^{(k)} \cup \lambda_{i_{m}}^{(k)}$}
 \label{l1l2}
\end{figure}

From the definition of $\varepsilon_{n}^{(k)}$, the curvature of the curve $\tfrac{1}{\alpha_{i}^{(k)}\varepsilon_{n}^{(k)}}\cdot\big(\sigma_{i_{m-1}}^{(k)}\cup\sigma_{i_{m}}^{(k)}\big)$ satisfies 
\begin{equation}\label{kappa_est_0-2}
   |\kappa(s)|\leq 2\sqrt{\beta_{n}^{(k)}}, \quad s\in[0,2],
\end{equation}
as shown in \eqref{curvature}.
By the Taylor expansion of the formula obtained as \eqref{tangent}, 
inequalities \eqref{eq: f'f_est}, and an estimate $\beta_n^{(k)}=\frac{k}{n^2} \leq \frac{1}{n}$, we have
\begin{align}\label{eq: l(x_1x_2)}
    \ell\big(\wideparen{\tilde{x}_{1}\tilde{x}_{2}}\,\,\,\big) 
    &= \int_{t}^{r}\sqrt{1+\big(\dot{f}_{\sqrt{\beta_{n}^{(k)}}}(u)\big)^{2}-2\kappa(u)f_{\sqrt{\beta_{n}^{(k)}}}(u)+\kappa^{2}(u)f_{\sqrt{\beta_{n}^{(k)}}}^{2}(u)}\,du\notag\\
    &\leq\int_{t}^{r}\left(1+\tfrac{1}{2}\big(\dot{f}_{\sqrt{\beta_{n}^{(k)}}}(u)\big)^{2}-\kappa(u)f_{\sqrt{\beta_{n}^{(k)}}}(u)+O\bigl((\beta_{n}^{(k)})^2\bigr)\right)\,du\notag\\
    &\leq (r-t)+\tfrac{\pi^2}{2}\beta_{n}^{(k)}(r-t)+2\beta_{n}^{(k)}(r-t)+O\bigl((\beta_{n}^{(k)})^2\bigr)(r-t)\notag\\
    &\leq (r-t)\big(1+\pi^{2}\beta_{n}^{(k)}+2\beta_{n}^{(k)}\big)\leq (r-t)\left(1+\tfrac{12}{n}\right)
\end{align}
for sufficiently large $n$.

Furthermore, we obtain
\begin{align*}
    |\tilde{x}_{1}-\tilde{x}_{2}|=
    & \bigl|\Gamma(r)-\Gamma(t)+f_{\sqrt{\beta_{n}^{(k)}}}(r)N(r)-f_{\sqrt{\beta_{n}^{(k)}}}(t)N(t)\bigr|\\
    \geq& |\Gamma(r)-\Gamma(t)|- \left(\bigl|f_{\sqrt{\beta_{n}^{(k)}}}(r)N(r)-f_{\sqrt{\beta_{n}^{(k)}}}(t)N(r)\bigr|\right.\\
    &\left.+\bigl|f_{\sqrt{\beta_{n}^{(k)}}}(t)N(r)-f_{\sqrt{\beta_{n}^{(k)}}}(t)N(t)\bigr|\right)\\
    =& |\Gamma(r)-\Gamma(t)| - \left(\bigl|f_{\sqrt{\beta_{n}^{(k)}}}(r)-f_{\sqrt{\beta_{n}^{(k)}}}(t)\bigr|+f_{\sqrt{\beta_{n}^{(k)}}}(t)|N(r)-N(t)|\right).
\end{align*}
To derive a lower bound of $|\Gamma(r)-\Gamma(t)|$, we utilize \eqref{kappa_est_0-2}. 

By the comparison theorem for planar curves (a direct consequence of Schur's theorem; see \cite[p.412]{doC}), when the curvature magnitude of a curve segment is uniformly bounded by $K$ (here $K=2\sqrt{\beta_{n}^{(k)}}$), the Euclidean distance between its endpoints is bounded below by the chord length of a comparison circular arc with constant curvature $K$ (and radius $R = 1/K$).

The central angle of the corresponding circular arc is $K(r-t)$. Thus, the chord length formula yields:
\[
|\Gamma(r) - \Gamma(t)| \geq \frac{2}{K} \sin\left(\frac{K(r-t)}{2}\right).
\]

Applying the Taylor expansion for the sine function, we obtain $\sin x \geq x - \frac{x^3}{6}$ (for $x \geq 0$). 
Since $r, t\in [0,2]$ with $r \geq t$, it follows that
\begin{align}\label{lower_bound_gamma_s-t}
   |\Gamma(r) - \Gamma(t)| &\geq \tfrac{2}{K} \left( \tfrac{K(r-t)}{2} - \tfrac{K^3(r-t)^3}{48} \right) \notag= (r-t) - \tfrac{K^2}{24}(r-t)^3\\
   &\geq (r-t)-\tfrac{2\beta_{n}^{(k)}}{3}(r-t)\geq  (r-t)\bigl(1-\tfrac{2}{3n}\bigr).
\end{align}

By the Mean Value Theorem applied at some $\xi \in (t,r)$, the Frenet--Serret formulas show that
\begin{align*}
    |N(r)-N(t)|=|\dot{N}(\xi)|(r-t)=|\kappa(\xi)|(r-t).
\end{align*}
Together with the Lipschitz continuity of the cosine function and the estimate \eqref{lower_bound_gamma_s-t}, it follows that
\begin{align}\label{eq: |x_1-x_2|}
    |\tilde{x}_{1}-\tilde{x}_{2}|\geq
    &\left(r-t\right)\bigl(1-\tfrac{2}{3n}\bigr)-\left(\tfrac{\sqrt{k}}{2n}|\cos (2 \pi r)-\cos (2 \pi t)|+\tfrac{\sqrt{k}}{n}|\kappa(\xi)|(r-t)\right) \notag\\
    \geq&(r-t)\bigl(1-\tfrac{2}{3n}\bigr)-\left(\tfrac{\pi\sqrt{k}}{n}(r-t)+\tfrac{\sqrt{k}}{n}\cdot\tfrac{2\sqrt{k}}{n}(r-t)\right)\notag \\
   \geq&(r-t)\big(1-\tfrac{2}{3n}-\tfrac{\pi}{\sqrt{n}}-\tfrac{2}{n}\big) \geq(r-t)\big(1-\tfrac{4}{\sqrt{n}}\big)
\end{align}
for sufficiently large $n$ since $1/n=o(1/\sqrt{n})$.

From \eqref{eq: l(x_1x_2)} and \eqref{eq: |x_1-x_2|}, we have
\begin{equation}\label{firstconclusion}
    \frac{\ell\big(\wideparen{x_{1} x_{2}}\big)}{\lvert x_{1}-x_{2}\rvert }=\frac{\ell\big(\wideparen{\tilde{x}_{1} \tilde{x}_{2}}\,\,\,\big)}{\lvert \tilde{x}_{1}-\tilde{x}_{2}\rvert } \leq \frac{1+\frac{12}{n}}{1-\frac{4}{\sqrt{n}}} \leq 1+\frac{5}{\sqrt{n}}
\end{equation} 
for sufficiently large $n$.

By a similar argument, the result also holds for the case $x_{1},\,x_{2}\in \lambda_{(i-1)_{M}}^{(k)} \cup \lambda_{i_{1}}^{(k)}$. For the reader's convenience, we provide a proof for this case below.

It suffices to show that the result is valid for any $x_{1}\in\lambda_{(i-1)_M}^{(k)}, x_{2}\in\lambda_{i_{1}}^{(k)}$. 
Assume $\tilde{x}_{1},\tilde{x}_{2}$ are points on $\tfrac{1}{\alpha_{i-1}^{(k)}\varepsilon_{n}^{(k)}}\cdot\bigl (\lambda_{(i-1)_M}^{(k)}\cup\lambda_{i_{1}}^{(k)}\bigr)$ corresponding to $ x_{1},\,x_{2} $ (see Figure~\ref{l1l2_appendix}). 
Let $\Gamma(s)$ be the arc-length parametrization of the curve
\[
    \tfrac{1}{\alpha_{i-1}^{(k)}\varepsilon_{n}^{(k)}}\cdot\bigl(\sigma_{(i-1)_M}^{(k)}\cup\sigma_{i_{1}}^{(k)}\bigr),\quad s\in\Bigl[0, 1+\tfrac{\alpha_{i}^{(k)}}{\alpha_{i-1}^{(k)}}\Bigr].
\]
Let the projection of $\tilde{x}_{1},\tilde{x}_{2}$ into $\tfrac{1}{\alpha_{i-1}^{(k)}\varepsilon_{n}^{(k)}}\cdot\bigl(\sigma_{(i-1)_M}^{(k)}\cup\sigma_{i_{1}}^{(k)}\bigr)$ be $\Gamma(t)$, $\Gamma(r)$, respectively. 
Then we have 
\[ 
    \begin{cases}
        \tilde{x}_{1}=\Gamma(t)+f_{\sqrt{\beta_{n}^{(k)}}}(t)N(t),\\
        \tilde{x}_{2}=\Gamma(r)-\tfrac{\alpha_{i}^{(k)}}{\alpha_{i-1}^{(k)}}f_{\sqrt{\beta_{n}^{(k)}}}\bigl(\tfrac{\alpha_{i-1}^{(k)}}{\alpha_{i}^{(k)}}(r-1)\bigr)N(r)\eqqcolon\Gamma(r)+ g_{\sqrt{\beta_{n}^{(k)}}}(r)N(r)
    \end{cases}
\]
for $0 \leq t \leq 1 \leq r \leq 1+\tfrac{\alpha_{i}^{(k)}}{\alpha_{i-1}^{(k)}}$.

\begin{figure}[htp]
 \centering
 \includegraphics[width=8cm]{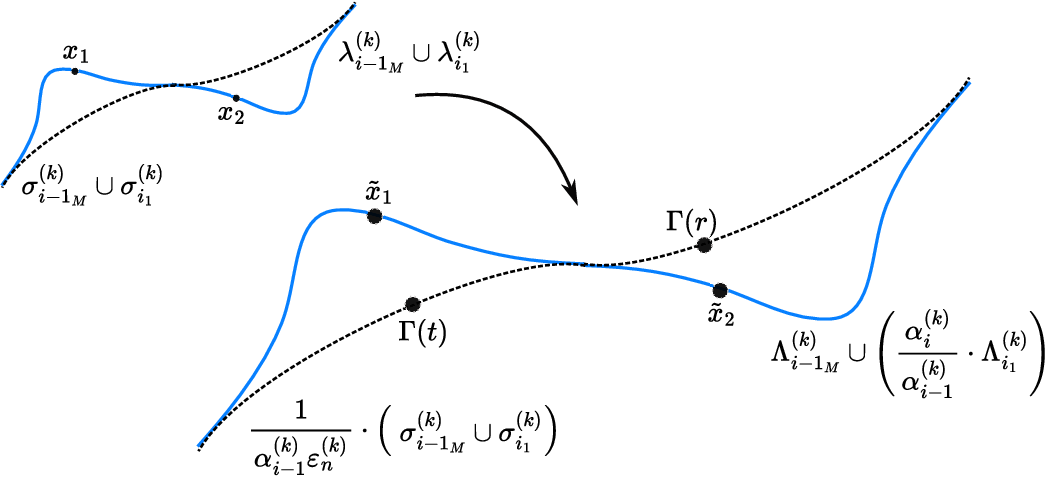}
 \caption{Adjacent subarcs $\sigma_{(i-1)_M}^{(k)} \cup \sigma_{i_1}^{(k)}$ and $ \lambda_{(i-1)_{M}}^{(k)} \cup \lambda_{i_{1}}^{(k)}$}
 \label{l1l2_appendix}
\end{figure}
From \eqref{curvature}, the curvature of the curve $\tfrac{1}{\alpha_{i}^{(k)}\varepsilon_{n}^{(k)}}\cdot\bigl(\sigma_{i_{m-1}}^{(k)}\cup\sigma_{i_{m}}^{(k)}\bigr)$ satisfies 
\[
    |\kappa(s)|\leq 2\sqrt{\beta_{n}^{(k)}}, \quad s\in\Bigl[0, 1+\tfrac{\alpha_{i}^{(k)}}{\alpha_{i-1}^{(k)}}\Bigr].
\]
By the estimate under the Taylor expansion as before and inequalities \eqref{eq: f'f_est}, we have
\begin{align}\label{l(x_1x_2)_appendix}
    \ell\bigl(\wideparen{\tilde{x}_{1}\tilde{x}_{2}}\,\,\,\bigr) =& \int_{t}^{1}\sqrt{1+\big(\dot{f}_{\sqrt{\beta_{n}^{(k)}}}(u)\big)^{2}-2\kappa(u)f_{\sqrt{\beta_{n}^{(k)}}}(u)+\kappa^{2}(u)f_{\sqrt{\beta_{n}^{(k)}}}^{2}(u)}\,du\notag\\
    &+\int_{1}^{r}\sqrt{1+\bigl(\dot{g}_{\sqrt{\beta_{n}^{(k)}}}(u)\bigr)^{2}-2\kappa(u)g_{\sqrt{\beta_{n}^{(k)}}}(u)+\kappa^{2}(u)g_{\sqrt{\beta_{n}^{(k)}}}^{2}(u)}\,du \notag\\
    \leq&\int_{t}^{1}\Bigl(1+\tfrac{1}{2}\big(\dot{f}_{\sqrt{\beta_{n}^{(k)}}}(u)\big)^{2}-\kappa(u)f_{\sqrt{\beta_{n}^{(k)}}}(u)+O\bigl((\beta_{n}^{(k)})^2\bigr)\Bigr)\,du\notag\\
    &+\int_{1}^{r}\Bigl(1+\tfrac{1}{2}\bigl(\dot{g}_{\sqrt{\beta_{n}^{(k)}}}(u)\bigr)^{2}-\kappa(u)g_{\sqrt{\beta_{n}^{(k)}}}(u)+O\bigl((\beta_{n}^{(k)})^2\bigr)\Bigr)\,du\notag\\
    \leq &(r-t)+\tfrac{\pi^2}{2}\beta_{n}^{(k)}(r-t)+4\beta_{n}^{(k)}(r-t)+O\bigl((\beta_{n}^{(k)})^2\bigr)(r-t)\notag\\
    \leq &(r-t)\left(1+\pi^{2}\beta_{n}^{(k)}+4\beta_{n}^{(k)}\right)\leq (r-t)\left(1+\tfrac{14}{n}\right)
\end{align}
for sufficiently large $n$.

Furthermore, 
\begin{align}\label{eq: tildex1x2_appendix}
     |\tilde{x}_{1}-\tilde{x}_{2}|
    =& \bigl|\Gamma(r)-\Gamma(t)+ g_{\sqrt{\beta_{n}^{(k)}}}(r)N(r)-f_{\sqrt{\beta_{n}^{(k)}}}(t)N(t)\bigr | \notag\\
\geq &\lvert \Gamma(r)-\Gamma(t)\rvert -\Bigl(\bigl| g_{\sqrt{\beta_{n}^{(k)}}}(r)N(r)-f_{\sqrt{\beta_{n}^{(k)}}}(t)N(r)\bigr|\notag\\
     &+\big|f_{\sqrt{\beta_{n}^{(k)}}}(t)N(r)-f_{\sqrt{\beta_{n}^{(k)}}}(t)N(t)\bigr|\Bigr) \notag\\
\geq &(r-t)\bigl(1-\tfrac{2}{3n}\bigr)-\Bigl(\Bigl|-\tfrac{\alpha_{i}^{(k)}}{\alpha_{i-1}^{(k)}}f_{\sqrt{\beta_{n}^{(k)}}}\bigl(\tfrac{\alpha_{i-1}^{(k)}}{\alpha_{i}^{(k)}}(r-1)\bigr)-f_{\sqrt{\beta_{n}^{(k)}}}\bigl(\tfrac{\alpha_{i-1}^{(k)}}{\alpha_{i}^{(k)}}(r-1)\bigr)\Bigr|\notag\\
     & +\Bigl|f_{\sqrt{\beta_{n}^{(k)}}}\bigl(\tfrac{\alpha_{i-1}^{(k)}}{\alpha_{i}^{(k)}}(r-1)\bigr)-f_{\sqrt{\beta_{n}^{(k)}}}(t)\Bigr|+\Bigl|f_{\sqrt{\beta_{n}^{(k)}}}(t)N(r)-f_{\sqrt{\beta_{n}^{(k)}}}(t)N(t)\Bigr| \Bigr)\notag\\
\eqqcolon& (r-t)\bigl(1-\tfrac{2}{3n}\bigr)-(I_{1}+I_{2}+I_{3}).
\end{align}

Applying Claim~\ref{slope}, and noting that $\frac{\alpha_{i-1}^{(k)}}{\alpha_{i}^{(k)}}\in [1,2]$ and $t\leq 1\leq r$, we bound the first term $I_1$ as
\begin{align}\label{I_1}
    I_{1} &=\Bigl |\tfrac{\alpha_{i}^{(k)}}{\alpha_{i-1}^{(k)}}+1\Bigr|\cdot f_{\sqrt{\beta_{n}^{(k)}}}\bigl(\tfrac{\alpha_{i-1}^{(k)}}{\alpha_{i}^{(k)}}(r-1)\bigr)\notag\\
    &\leq 8 \sqrt{\beta_{n}^{(k)}} \Bigl( \tfrac{\alpha_{i}^{(k)}}{\alpha_{i-1}^{(k)}}+1\Bigr) \tfrac{\alpha_{i-1}^{(k)}}{\alpha_{i}^{(k)}}(r-1) \notag\\
          &\leq \tfrac{8}{\sqrt{n}}(r-1)\Bigl(  1+\tfrac{\alpha_{i-1}^{(k)}}{\alpha_{i}^{(k)}}\Bigr)  \leq \tfrac{24}{\sqrt{n}}(r-t).
\end{align}
Moreover, for the second term $I_2$, we have
\begin{align}\label{I_2}
    I_{2} &= \tfrac{\sqrt{k}}{2n}\Bigl\lvert \cos\bigl(2\pi(r-1)\cdot\tfrac{\alpha_{i-1}^{(k)}}{\alpha_{i}^{(k)}}\bigr)-\cos\left(2\pi(t-1)\right)\Bigr\rvert  \notag\\
    &\leq \tfrac{\pi}{\sqrt{n}}\Bigl\lvert (r-1)\tfrac{\alpha_{i-1}^{(k)}}{\alpha_{i}^{(k)}}-(t-1)\Bigr\rvert \notag\\ 
    &\leq \tfrac{\pi}{\sqrt{n}}(2(r-1)+1-t) 
    \leq \tfrac{2\pi}{\sqrt{n}}(r-t).
\end{align}

Finally, by the Mean Value Theorem and the Frenet--Serret formulas, the third term $I_3$ is estimated as follows:
\begin{align}\label{I_3}
    I_{3} = f_{\sqrt{\beta_{n}^{(k)}}}(t)|\kappa(\xi)|(r-t) \leq \sqrt{\beta_{n}^{(k)}}\cdot 2\sqrt{\beta_{n}^{(k)}}(r-t) \leq \tfrac{2}{n}(r-t).
\end{align}

From inequalities \eqref{eq: tildex1x2_appendix}, \eqref{I_1}, \eqref{I_2}, and \eqref{I_3}, we deduce that 
\begin{align}\label{tildex1x2}
    |\tilde{x}_{1}-\tilde{x}_{2}|\geq (r-t)\bigl(1-\tfrac{2}{3n}-\tfrac{24}{\sqrt{n}}-\tfrac{2\pi}{\sqrt{n}}-\tfrac{2}{n}\bigr) \geq (r-t)\bigl(1-\tfrac{31}{\sqrt{n}}\bigr)
\end{align}
for sufficiently large $n$. 
Therefore, combining \eqref{l(x_1x_2)_appendix} and \eqref{tildex1x2}, we conclude
\begin{align*}
    \frac{\ell\big(\wideparen{x_{1}x_{2}}\big)}{|x_{1}-x_{2}|} = \frac{\ell\big(\wideparen{\tilde{x}_{1}\tilde{x}_{2}}\,\,\,\big)}{|\tilde{x}_{1}-\tilde{x}_{2}|} \leq \frac{1+\frac{14}{n}}{1-\frac{31}{\sqrt{n}}} \leq 1+\frac{32}{\sqrt{n}}
\end{align*}
for sufficiently large $n$.

Combined with \eqref{firstconclusion}, this completes the proof of Lemma~\ref{x1x2_in_small_bump}.
\end{proof}

\begin{lemma}\label{x1-x2<epsilon:label}
Let $x_{i}=x_{i}^{(k)}\in\gamma_{n}^{(k)}$ for $i=1,2$, and let $x_{i}^{(j)}$ be the projection of $x_{i}^{(k)}$ on $\gamma_{n}^{(j)}$ defined inductively for $1\leq j\leq k$. 
If $|x_{1}^{(k-1)}-x_{2}^{(k-1)}|<\tfrac{\varepsilon_{n}^{(k-1)}}{2}$, then $x_1=x_{1}^{(k)}$ and $x_2=x_{2}^{(k)}$ lie on adjacent bumps (or on the same bump) of $\gamma_{n}^{(k)}$, namely, $x_{1}^{(k-1)}$ and $x_{2}^{(k-1)}$ lie on adjacent (or identical) partition subarcs of $\gamma_{n}^{(k-1)}$.
\end{lemma}
\begin{proof}
Indeed, suppose to the contrary that $x_{1}^{(k)}$ and $x_{2}^{(k)}$ do not lie on adjacent bumps of $\gamma_n^{(k)}$. 
Then $x_{1}^{(k-1)}$ and $x_{2}^{(k-1)}$ are separated by at least one complete partition subarc of $\gamma_n^{(k-1)}$.

Tracing the projections of these two points through the preceding generations, there exists a smallest index $K\le k-1$ such that $x_{1}^{(K)}$ and $x_{2}^{(K)}$ are separated by at least one complete partition subarc of $\gamma_n^{(K)}$, whereas $x_{1}^{(K-1)}$ and $x_{2}^{(K-1)}$ lie on adjacent (or identical) partition subarcs of $\gamma_n^{(K-1)}$.

Since $x_{1}^{(K)}$ and $x_{2}^{(K)}$ are separated by a complete partition subarc of $\gamma_n^{(K)}$, we have
\[
\ell\big(\wideparen{x_{1}^{(K)}x_{2}^{(K)}}\big)\geq \varepsilon_n^{(K)}.
\]
Moreover, Lemma~\ref{x1x2_in_small_bump} and Claim~\ref{projection} imply that
\begin{align}\label{x1k-1-x2k-1}
   |x_{1}^{(K-1)}-x_{2}^{(K-1)}|\geq \frac{\ell\big(\wideparen{x_{1}^{(K-1)}x_{2}^{(K-1)}}\big)}{1+\frac{32}{\sqrt n}}\geq \frac{\ell\big(\wideparen{x_{1}^{(K)}x_{2}^{(K)}}\big)}{\left(1+\frac{32}{\sqrt n}\right)\bigl(1+8\beta_n^{(K)}\bigr)}.
\end{align}

Combining the estimates \eqref{eq: |x_1-x_2|} and \eqref{tildex1x2} with the fact that arc-length parametrization is $1$-Lipschitz, we obtain
\begin{align}\label{x1K_x2K}
   \frac{|x_{1}^{(K)}-x_{2}^{(K)}|}{|x_{1}^{(K-1)}-x_{2}^{(K-1)}|}\geq 1-\frac{31}{\sqrt{n}}
\end{align}
for sufficiently large $n$.
Thus, by \eqref{x1k-1-x2k-1} and \eqref{x1K_x2K}, we have
\begin{align*}
   |x_{1}^{(K)}-x_{2}^{(K)}|\geq |x_{1}^{(K-1)}-x_{2}^{(K-1)}|\left(1-\tfrac{31}{\sqrt n}\right)\geq \frac{\left(1-\tfrac{31}{\sqrt n}\right)\varepsilon_n^{(K)}}{\left(1+\frac{32}{\sqrt n}\right)\bigl(1+8\beta_n^{(K)}\bigr)}\geq \frac{3\varepsilon_n^{(K)}}{4}
\end{align*}
for sufficiently large $n$.

On the other hand, by the triangle inequality, Claim~\ref{D_gammak_to_gamman}, and the assumption, it follows that
\begin{align*}
   |x_{1}^{(K)}-x_{2}^{(K)}|&\leq |x_{1}^{(K)}-x_{1}^{(k-1)}|+|x_{2}^{(K)}-x_{2}^{(k-1)}|+|x_{1}^{(k-1)}-x_{2}^{(k-1)}|\\
   & \leq \tfrac{8\varepsilon_n^{(K)}}{\sqrt{n}}+\tfrac{\varepsilon_{n}^{(k-1)}}{2}\leq \tfrac{2\varepsilon_n^{(K)}}{3}
\end{align*}
for sufficiently large $n$, which contradicts the preceding lower bound.
Therefore, $x_{1}^{(k)}$ and $x_{2}^{(k)}$ must lie on adjacent bumps (or on the same bump) of $\gamma_n^{(k)}$.
\end{proof}
\begin{lemma}\label{x1x2_in_big_bump}
    For any sufficiently large $n$, let $x_1, \, x_2$ lie on the curve $\gamma_{n}^{(k)}$ with $1\leq k\leq n$. Then
    \[
        \max_{w\in\wideparen{x_{1}x_{2}}}\frac{|x_{1}-w|+|w-x_{2}|}{|x_{1}-x_{2}|}\leq 1+\frac{81}{\sqrt{n}}.
    \]
\end{lemma}
\begin{proof}
Let $x_{i}=x_{i}^{(k)}\in\gamma_{n}^{(k)}$ for $i=1,2$, and let $x_{i}^{(j)}$ be the projection of $x_{i}^{(k)}$ on $\gamma_{n}^{(j)}$ defined inductively for $1\leq j\leq k$. 
If $|x_{1}^{(k-1)}-x_{2}^{(k-1)}|<\tfrac{\varepsilon_{n}^{(k-1)}}{2}$, then $x_1=x_{1}^{(k)}$ and $x_2=x_{2}^{(k)}$ lie on adjacent bumps of $\gamma_{n}^{(k)}$ by Lemma~\ref{x1-x2<epsilon:label}. 
In this case, by Lemma~\ref{x1x2_in_small_bump}, we have
\[
    \max_{w\in\wideparen{x_{1}x_{2}}}\frac{|x_{1}-w|+|w-x_{2}|}{|x_{1}-x_{2}|}\leq 
    \frac{\ell\big(\wideparen{x_{1}x_{2}}\big)}{\lvert x_{1}-x_{2}\rvert } \leq 1+\frac{32}{\sqrt{n}}.
\]

Otherwise, there must exist $1\leq j\leq k-1$ such that
\[
    \begin{cases}
        |x_{1}^{(j)}-x_{2}^{(j)}|\geq\frac{\varepsilon_{n}^{(j)}}{2},\\
        |x_{1}^{(j-1)}-x_{2}^{(j-1)}|<\frac{\varepsilon_{n}^{(j-1)}}{2}.
    \end{cases}
\]
Note that if the first condition holds for $j=1$, we ignore the second condition in this case. 
Since $|x_{1}^{(j-1)}-x_{2}^{(j-1)}|<\frac{\varepsilon_{n}^{(j-1)}}{2}$, by Lemma~\ref{x1-x2<epsilon:label} we deduce that $x_{1}^{(j)}$ and $x_{2}^{(j)}$ lie on adjacent bumps $\lambda_{i_{m-1}}^{(j-1)}\cup\lambda_{i_{m}}^{(j-1)}$ or $\lambda_{{(i-1)}_M}^{(j-1)}\cup\lambda_{i_{1}}^{(j-1)}$ of $\gamma_{n}^{(j)}$. 
By Lemma~\ref{x1x2_in_small_bump} again, we obtain
\begin{align}\label{j}
    \max_{w^{(j)}\in\wideparen{x_{1}^{(j)}x_{2}^{(j)}}}\frac{\lvert x_{1}^{(j)}-w^{(j)}\rvert +\lvert w^{(j)}-x_{2}^{(j)}\rvert }{\lvert x_{1}^{(j)}-x_{2}^{(j)}\rvert }\leq\frac{\ell\big(\wideparen{x_{1}^{(j)}x_{2}^{(j)}}\big)}{\lvert x_{1}^{(j)}-x_{2}^{(j)}\rvert }\leq 1+\frac{32}{\sqrt{n}}.
\end{align}

For any $w$ lying on $\wideparen{x_{1}x_{2}}$, by Claim~\ref{D_gammak_to_gamman}, we have
\begin{align*}
    &\quad \frac{\lvert x_{1}-w\rvert + \lvert w-x_{2}\rvert}{\lvert x_{1}-x_{2}\rvert} \\
    &\leq \frac{\lvert x_{1}-x_{1}^{(j)}\rvert + \lvert x_{1}^{(j)}-w^{(j)}\rvert  + \lvert w^{(j)}-w\rvert + \lvert w-w^{(j)}\rvert + \lvert w^{(j)}-x_{2}^{(j)}\rvert + \lvert x_{2}^{(j)}-x_{2}\rvert}{\lvert x_{1}^{(j)}-x_{2}^{(j)}\rvert  - \big(\lvert x_{1}^{(j)}-x_{1}\rvert + \lvert x_{2}^{(j)}-x_{2}\rvert \big)} \\
    &\leq \frac{\tfrac{16\varepsilon_{n}^{(j)}}{\sqrt{n}}+ \lvert x_{1}^{(j)}-w^{(j)}\rvert  + \lvert w^{(j)}-x_{2}^{(j)}\rvert}{\lvert x_{1}^{(j)}-x_{2}^{(j)}\rvert - \tfrac{8\varepsilon_{n}^{(j)}}{\sqrt{n}}}.
\end{align*}
Together with the condition $\lvert x_{1}^{(j)}-x_{2}^{(j)}\rvert \geq\frac{\varepsilon_{n}^{(j)}}{2}$, Lemma~\ref{x1x2_in_small_bump}, and the inequality~\eqref{j}, it follows that
\[
    \begin{aligned}
        &\quad\max_{w\in\wideparen{x_{1}x_{2}}}\frac{|x_{1}-w|+|w-x_{2}|}{|x_{1}-x_{2}|}\\
        &\leq \max_{w^{(j)}\in\wideparen{x_{1}^{(j)}x_{2}^{(j)}}}\frac{\frac{32}{\sqrt{n}}\lvert x_1^{(j)}-x_2^{(j)}\rvert +\big(1+\frac{32}{\sqrt{n}}\big)\lvert x_1^{(j)}-x_2^{(j)}\rvert }{\lvert x_1^{(j)}-x_2^{(j)}\rvert -\frac{16}{\sqrt{n}}\lvert x_1^{(j)}-x_2^{(j)}\rvert } = \frac{1+\frac{64}{\sqrt{n}}}{1-\frac{16}{\sqrt{n}}}\leq 1+\frac{81}{\sqrt{n}}
    \end{aligned}
\]
for sufficiently large $n$.
\end{proof}

\begin{lemma}\label{x1x2_in_union_of_big_bumps}
For any fixed sufficiently large $N$, let $x_{1},\,x_{2}$ lie on the curve $\bigcup\limits_{n\geq N}\gamma_{n}^{(n)}$. 
Then
    \[
        \max\limits_{w\in\wideparen{x_{1}x_{2}}}\frac{|x_{1}-w|+|w-x_{2}|}{|x_{1}-x_{2}|}\leq 1+\frac{114}{\sqrt{N}}.
    \]
\end{lemma}
\begin{proof}
Recall that we denote $\big(\tfrac{1}{2^{n}},0\big)$ by $M_{n}$. 
Assume that $x_{1}\in\gamma_{n_{1}}^{(n_{1})}, x_{2}\in\gamma_{n_{2}}^{(n_{2})}$ with $n_1\geq n_2\geq N$.
     
\textbf{Case 1.} $n_{1}=n_{2}$: In this case, the claim is an immediate consequence of Lemma~\ref{x1x2_in_big_bump}.
     
\textbf{Case 2.} $n_{1}=n_{2}+1$: We distinguish two possibilities according to the position of $w$.
     
If $w\in\wideparen{x_{1}M_{n_{2}}}$, then
\begin{align}\label{case2_1}
    |x_{1}-w|+|w-x_{2}|
    &\leq \lvert x_{1}-w\rvert +\lvert w-M_{n_{2}}\rvert +\lvert M_{n_{2}}-x_{2}\rvert  \notag\\
    &\leq \big(1+\tfrac{81}{\sqrt{n_1}}\big)\lvert x_{1}-M_{n_{2}}\rvert +\lvert M_{n_{2}}-x_{2}\rvert \notag\\
    &\leq \big(1+\tfrac{81}{\sqrt{n_1}}\big)\left(\lvert x_{1}-M_{n_{2}}\rvert +\lvert M_{n_{2}}-x_{2}\rvert \right),
\end{align}
where the second inequality follows from Lemma~\ref{x1x2_in_big_bump}.

If $w\in\wideparen{M_{n_{2}}x_{2}}$, then
\begin{align}\label{case2_2}
    |x_{1}-w|+|w-x_{2}|
    &\leq \lvert x_{1}-M_{n_2}\rvert +\lvert M_{n_{2}}-w\rvert +\lvert w-x_{2}\rvert  \notag\\
    &\leq \lvert x_{1}-M_{n_{2}}\rvert +\big(1+\tfrac{81}{\sqrt{n_2}}\big)\lvert M_{n_{2}}-x_{2}\rvert \notag\\
    &\leq \big(1+\tfrac{81}{\sqrt{n_2}}\big)\left(\lvert x_{1}-M_{n_{2}}\rvert +\lvert M_{n_{2}}-x_{2}\rvert \right),
\end{align}
where, once again, the second inequality follows from Lemma~\ref{x1x2_in_big_bump}.

As in the proof of Lemma~\ref{x1x2_in_big_bump}, there exist indices $2\leq j,\,k\leq n-1$ such that
\[
    \begin{cases}
        |x_{1}^{(j)}-M_{n_{2}}|\geq\frac{\varepsilon_{n_{1}}^{(j)}}{2},\\
        |x_{1}^{(j-1)}-M_{n_{2}}|<\frac{\varepsilon_{n_{1}}^{(j-1)}}{2},
    \end{cases}
\]
and
\[
    \begin{cases}
        |x_{2}^{(k)}-M_{n_{2}}|\geq\frac{\varepsilon_{n_{2}}^{(k)}}{2},\\
        |x_{2}^{(k-1)}-M_{n_{2}}|<\frac{\varepsilon_{n_{2}}^{(k-1)}}{2}.
    \end{cases}
\]
Then, by Claim~\ref{D_gammak_to_gamman}, 
\begin{align}\label{eq: x1-Mn2+Mn2-x2}
        & \lvert x_{1}-M_{n_{2}}\rvert +\lvert M_{n_{2}}-x_{2}\rvert \notag\\
   \leq & \lvert x_{1}-x_{1}^{(j)}\rvert +\lvert x_{1}^{(j)}-M_{n_{2}}\rvert +\lvert M_{n_{2}}-x_{2}^{(k)}\rvert +\lvert x_{2}^{(k)}-x_{2}\rvert  \notag\\
   \leq & \tfrac{4\varepsilon_{n_{1}}^{(j)}}{\sqrt{n_1}}+\lvert x_{1}^{(j)}-M_{n_{2}}\rvert +\lvert M_{n_{2}}-x_{2}^{(k)}\rvert +\tfrac{4\varepsilon_{n_{2}}^{(k)}}{\sqrt{n_2}}.
\end{align}

By the choice of $j$ and $k$, together with the inequality~\eqref{eq: x1-Mn2+Mn2-x2}, we obtain
\begin{align}\label{eq: x1-Mn2+Mn2-x2_est}
        & \lvert x_{1}-M_{n_{2}}\rvert +\lvert M_{n_{2}}-x_{2}\rvert \notag\\
   \leq &\big(1+\tfrac{8}{\sqrt{n_{1}}}\big)\lvert x_{1}^{(j)}-M_{n_{2}}\rvert  +\big(1+ \tfrac{8}{\sqrt{n_{2}}}\big)\lvert x_{2}^{(k)}-M_{n_{2}}\rvert \notag\\
   \leq & \big(1+\tfrac{8}{\sqrt{N}}\big)\big(\lvert x_{1}^{(j)}-M_{n_{2}}\rvert +\lvert M_{n_{2}}-x_{2}^{(k)}\rvert \big)\notag\\
   \leq & \big(1+\tfrac{8}{\sqrt{N}}\big)\big(1+\tfrac{8}{\sqrt{N}}\big)\big(\operatorname{Re}\lvert x_{1}^{(j)}-M_{n_{2}}\rvert +\operatorname{Re}\lvert M_{n_{2}}-x_{2}^{(k)}\rvert \big)\notag\\
   =    &\big(1+\tfrac{8}{\sqrt{N}}\big)^2\operatorname{Re}\big(x_{2}^{(k)}-x_{1}^{(j)}\big).
\end{align}
Here we have applied the estimates
\[
    \lvert x_{1}^{(j)}-M_{n_{2}}\rvert \leq\big(1+\tfrac{8}{\sqrt{N}}\big)\operatorname{Re}\lvert x_{1}^{(j)}-M_{n_{2}}\rvert ,
\]
and
\[
    \lvert M_{n_{2}}-x_{2}^{(k)}\rvert \leq\big(1+\tfrac{8}{\sqrt{N}}\big)\operatorname{Re}\lvert M_{n_{2}}-x_{2}^{(k)}\rvert ,
\]
which follow from Claim~\ref{slope}.

On the other hand, together with the choice of $j$ and $k$, and Claim~\ref{D_gammak_to_gamman}, it follows that
\begin{align}\label{case2_denom}
    \lvert x_{1}-x_{2}\rvert  &\geq \lvert x_{1}^{(j)}-x_{2}^{(k)}\rvert -\big(\lvert x_{2}^{(k)}-x_{2}\rvert +\lvert x_{1}^{(j)}-x_{1}\rvert \big)\notag\\
    &\geq  \lvert x_{1}^{(j)}-x_{2}^{(k)}\rvert -\big(\tfrac{4\varepsilon_{n_{2}}^{(k)}}{\sqrt{n_{2}}}+\tfrac{4\varepsilon_{n_{1}}^{(j)}}{\sqrt{n_{1}}}\big)\notag \\
    &\geq \lvert x_{1}^{(j)}-x_{2}^{(k)}\rvert -\tfrac{8}{\sqrt{N}}\big(\lvert x_{2}^{(k)}-M_{n_{2}}\rvert +\lvert x_{1}^{(j)}-M_{n_{2}}\rvert \big) \notag\\
    &\geq \lvert x_{1}^{(j)}-x_{2}^{(k)}\rvert -\tfrac{8}{\sqrt{N}}\big(1+\tfrac{8}{\sqrt{N}}\big)\big(\operatorname{Re}\lvert x_{2}^{(k)}-M_{n_{2}}\rvert +\operatorname{Re}\lvert x_{1}^{(j)}-M_{n_{2}}\rvert \big) \notag\\
    &\geq \lvert x_{1}^{(j)}-x_{2}^{(k)}\rvert -\tfrac{16}{\sqrt{N}}\lvert x_{1}^{(j)}-x_{2}^{(k)}\rvert \notag \\
    &\geq \big(1-\tfrac{16}{\sqrt{N}}\big)\operatorname{Re}\big(x_{2}^{(k)}-x_{1}^{(j)}\big).
\end{align}
Combining \eqref{case2_1}, \eqref{case2_2}, \eqref{eq: x1-Mn2+Mn2-x2_est} and \eqref{case2_denom}, we conclude
\begin{align*}
    \max_{w\in\wideparen{x_1x_2}}\frac{\lvert x_{1}-w\rvert +\lvert w-x_{2}\rvert }{\lvert x_{1}-x_{2}\rvert } \leq \frac{\big(1+\frac{81}{\sqrt{N}}\big)\big(1+\frac{8}{\sqrt{N}}\big)^2}{1-\frac{16}{\sqrt{N}}} \leq 1+\frac{114}{\sqrt{N}}
\end{align*}
for sufficiently large $N$.

\textbf{Case 3.} $n_{1}> n_{2}+1$: Assume $w\in\gamma_{n_{3}}^{(n_{3})}$ for $n_{1}\geq n_{3}\geq n_{2}$. 

The definition of $\varepsilon_{n}^{(1)}$ implies that
\[
\varepsilon_n^{(1)}\leq \tfrac{1}{2^n\cdot 2^2}.
\]
Then by Claim~\ref{D_gammak_to_gamman}, it follows that 
\begin{align*}
     &\frac{\lvert x_{1}-w\rvert +\lvert w-x_{2}\rvert }{\lvert x_{1}-x_{2}\rvert } \\
\leq &\frac{\lvert x_{1}-x_{1}^{(1)}\rvert +\lvert x_{2}-x_{2}^{(1)}\rvert +2\lvert w-w^{(1)}\rvert+\lvert x_{1}^{(1)}-w^{(1)}\rvert +\lvert w^{(1)}-x_{2}^{(1)}\rvert}{\lvert x_{1}^{(1)}-x_{2}^{(1)}\rvert-\lvert x_{1}-x_{1}^{(1)}\rvert-\lvert x_{2}-x_{2}^{(1)}\rvert } \\     
\leq &\frac{\frac{1}{2^{n_{1}}\sqrt{n_{1}}}+\frac{1}{2^{n_{2}}\sqrt{n_{2}}}+\frac{2}{2^{n_{3}}\sqrt{n_{3}}}+\lvert x_{1}^{(1)}-w^{(1)}\rvert +\lvert w^{(1)}-x_{2}^{(1)}\rvert }{\lvert x_{1}^{(1)}-x_{2}^{(1)}\rvert -\big(\frac{1}{2^{n_{1}}\sqrt{n_{1}}}+\frac{1}{2^{n_{2}}\sqrt{n_{2}}}\big)}.
\end{align*}
Applying Claim~\ref{asymconf_gamma1} and noting that $\lvert x_{1}^{(1)}-x_{2}^{(1)}\rvert \geq \frac{1}{2^{n_2+1}}$, we obtain
\begin{align*}
     \frac{\lvert x_{1}-w\rvert +\lvert w-x_{2}\rvert }{\lvert x_{1}-x_{2}\rvert }
     &\leq \frac{\frac{4}{2^{n_{2}}\sqrt{n_{2}}}+\big(1+\frac{\pi}{n_{2}}\big)\lvert x_{1}^{(1)}-x_{2}^{(1)}\rvert }{\lvert x_{1}^{(1)}-x_{2}^{(1)}\rvert -\frac{2}{2^{n_{2}}\sqrt{n_{2}}}} \\
     &\leq \frac{\frac{8}{\sqrt{n_2}}\lvert x_{1}^{(1)}-x_{2}^{(1)}\rvert +\big(1+\frac{\pi}{n_{2}}\big)\lvert x_{1}^{(1)}-x_{2}^{(1)}\rvert }{\lvert x_{1}^{(1)}-x_{2}^{(1)}\rvert -\frac{4}{\sqrt{n_2}}\lvert x_{1}^{(1)}-x_{2}^{(1)}\rvert }\\
     &=\frac{1 + \frac{\pi}{n_{2}} + \frac{8}{\sqrt{n_{2}}}}{1 - \frac{4}{\sqrt{n_{2}}}} \leq 1 + \frac{13}{\sqrt{n_{2}}} \leq 1 + \frac{13}{\sqrt{N}}.
\end{align*}
Combining Cases 1, 2, and 3, we complete the proof of the lemma.
\end{proof}

\subsection{Verification}
Theorem~\ref{main1} stated in the introduction follows from Theorems~\ref{not_asymp_smooth}, \ref{asymp_conf}, and \ref{chord-arc}, established in this subsection.

\begin{theorem}\label{not_asymp_smooth}
    The curve $\gamma$ constructed in Section~\ref{counterexample} is not asymptotically smooth.
\end{theorem}
\begin{proof}
We begin by recalling that $\gamma_{n}^{(k-1)}=\bigcup_{i,\, m}\sigma_{i_m}^{(k-1)}$, where the arcs $\sigma_{i_m}^{(k-1)}$ form a partition of $\gamma_{n}^{(k-1)}$. 
Consequently, we obtain
\begin{align}\label{eq: ratioinf}
    \frac{\ell\big(\gamma_{n}^{(k)}\big)}{\ell\big(\gamma_{n}^{(k-1)}\big)}
    &=\frac{\ell\big(\bigcup_{i,\, m}\lambda_{i_m}^{(k-1)}\big)}{\ell\big(\bigcup_{i,\, m}\sigma_{i_m}^{(k-1)}\big)}
     =\frac{\sum_{i,\, m} \ell\big(\lambda_{i_m}^{(k-1)}\big)}{\sum_{i,\, m} \ell\big(\sigma_{i_m}^{(k-1)}\big)} \geq \inf_{i,\, m} \frac{\ell\big(\lambda_{i_{m}}^{(k-1)}\big)}{\ell\big(\sigma_{i_{m}}^{(k-1)}\big)}.
\end{align}
Furthermore, by the construction of $\gamma$ together with Lemma~\ref{lengthest_of_embeded_curve}, we have
\begin{align}\label{eq: lratio}
    \frac{\ell\big(\lambda_{i_{m}}^{(k-1)}\big)}{\ell\big(\sigma_{i_{m}}^{(k-1)}\big)}
    =\ell\big(\Lambda_{i_{m}}^{(k-1)}\big)\geq 1+\beta_{n}^{(k-1)}.
\end{align}
In addition, a direct computation yields
\begin{align}\label{eq: length_est_gamma_n^1}
    \frac{1}{2^n}\left(1+\frac{1}{n^2}\right)\leq \ell\big(\gamma_n^{(1)}\big)\leq\frac{1}{2^n}\left(1+\frac{4}{n^2}\right).
\end{align}

Therefore, combining \eqref{eq: ratioinf}, \eqref{eq: lratio}, and \eqref{eq: length_est_gamma_n^1}, we deduce that for sufficiently large $n$,
\begin{align*}
    \ell\big(\gamma_n^{(n)}\big)
    &=\prod_{k=2}^n\tfrac{\ell\big(\gamma_n^{(k)}\big)}{\ell\big(\gamma_n^{(k-1)}\big)}\cdot \ell\big(\gamma_n^{(1)}\big)\\
    &\geq\prod_{k=2}^n\big(1+\beta_n^{(k-1)}\big)\cdot\tfrac{1}{2^n}\big(1+\tfrac{1}{n^2}\big)\geq\tfrac{1}{2^n}\exp\left(\sum_{k=2}^n\log\big(1+\beta_n^{(k-1)}\big)\right)\\
    &\geq\tfrac{1}{2^n}\exp\left(\tfrac{1}{2}\sum_{k=2}^n\tfrac{k-1}{n^2}\right)=\tfrac{1}{2^n}\exp\left(\tfrac{1}{2}\tfrac{n(n-1)}{2n^2}\right)\geq\tfrac{1}{2^n}\cdot e^{\tfrac{1}{5}},
\end{align*}
where in the last line we use the estimate $\log(1+x)\geq \tfrac{x}{2}$ for sufficiently small $x$.

Now, consider the points $x_1=\big(\tfrac{1}{2^n},0\big)\in\gamma$ and $x_2=\left(\tfrac{1}{2^{n-1}},0\right)\in\gamma$. Then
\[
    \frac{\ell\big(\wideparen{x_1x_2}\big)}{\lvert x_1-x_2\rvert }
    =\frac{\ell\big(\gamma_n^{(n)}\big)}{\lvert x_1-x_2\rvert }
    \geq e^{\frac{1}{5}},
\]
which contradicts the asymptotically smooth condition. 
Hence, the curve $\gamma$ cannot be asymptotically smooth.
\end{proof}

\begin{theorem}\label{asymp_conf}
    The curve $\gamma$ constructed in Section~\ref{counterexample} is asymptotically conformal.
\end{theorem}
\begin{proof}
For any $\varepsilon>0$, choose $N$ sufficiently large so that $\frac{114}{\sqrt{N}}<\varepsilon$ and all the preceding lemmas and claims apply for $n \geq N$. 
Fix such an $\varepsilon$ and $N$. 
Since
\[
    C_1\cup C_2\cup C_3\cup\left(\bigcup_{n=1}^{N+1}\gamma_n^{(n)}\right)
\]
is a $C^1$ curve, there exists $\delta_1>0$ such that for all $x_1,\,x_2$ on this curve with $\lvert x_1-x_2\rvert<\delta_1$, one has
\[
    \max_{w\in\wideparen{x_1x_2}}
    \frac{\lvert x_1-w\rvert+\lvert w-x_2\rvert}{\lvert x_1-x_2\rvert}
    <1+\varepsilon.
\]

Suppose $x_1\in C_1$, $x_2\in\gamma_n^{(n)}$ with $\lvert x_1-x_2\rvert <\tfrac{1}{2^N}$ for some $n\geq N$. 

If $w\in\gamma_{n_1}^{(n_1)}$ with $n_1\geq n$, then by Claim~\ref{D_gammak_to_gamman} and the definition of $\varepsilon_{n}^{(1)}$, we obtain
\begin{align}\label{continued}
        &\lvert x_{1}-w\rvert +\lvert w-x_{2}\rvert \notag \\
    \leq& |x_{1}-w^{(1)}|+|w^{(1)}-w|+|w-w^{(1)}|+|w^{(1)}-x_{2}^{(1)}|+|x_{2}^{(1)}-x_{2}|\notag\\
    \leq& \frac{3}{2^{n} \sqrt{n}}+|x_{1}-w^{(1)}|+|w^{(1)}-x_{2}^{(1)}|.
\end{align}
Since $x_2\in\gamma_n^{(n)}$ implies 
\[
    \lvert x_{1}-x_{2}^{(1)}\rvert\geq \operatorname{Re}\big( x_{2}^{(1)}-x_1\big)  \geq \frac{1}{2^{n}},
\] 
we further continue from \eqref{continued} so that
\begin{align}
        &\lvert x_{1}-w\rvert +\lvert w-x_{2}\rvert \notag \\
    \leq& \tfrac{3}{\sqrt{n}}\lvert x_{1}-x_{2}^{(1)}\rvert +\lvert x_{1}-w^{(1)}\rvert +\lvert w^{(1)}-x_{2}^{(1)}\rvert  \notag\\
    \leq& \big(1+\tfrac{3}{\sqrt{N}}\big)\big(\lvert x_{1}-M_{n_1}\rvert +\lvert M_{n_1}-w^{(1)}\rvert +\lvert w^{(1)}-x_{2}^{(1)}\rvert \big)\notag\\
    \leq&\big(1+\tfrac{3}{\sqrt{N}}\big)\left(\big(1+\tfrac{1}{N}\big) \operatorname{Re}(M_{n_1}-x_{1})+\big(1+\tfrac{\pi}{N}\big)\lvert x_{2}^{(1)}-M_{n_1}\rvert \right)\label{eq: semi}\\
    \leq& \big(1+\tfrac{3}{\sqrt{N}}\big)\big(1+\tfrac{\pi}{N}\big)\big(\operatorname{Re}(M_{n_1}-x_{1})+\big(1+\tfrac{8}{\sqrt{N}}\big) \operatorname{Re}(x_{2}^{(1)}-M_{n_1})\big)\label{eq: |x_2-M_n1|}\\
    \leq& \big(1+\tfrac{3}{\sqrt{N}}\big)\big(1+\tfrac{8}{\sqrt{N}}\big)\big(1+\tfrac{\pi}{N}\big) \operatorname{Re}(x_{2}^{(1)}-x_{1})\notag\\
    \leq& \big(1+\tfrac{12}{\sqrt{N}}\big) \operatorname{Re}(x_{2}^{(1)}-x_{1}),\label{eq: lem2_1_num}
\end{align}
where inequality \eqref{eq: semi} is obtained by a direct estimate on the semicircle and the use of Claim~\ref{asymconf_gamma1}, while inequality \eqref{eq: |x_2-M_n1|} follows from Claim~\ref{slope}.

Moreover, if $w\in C_1$, then by a parallel argument with the estimates $\lvert x_{2}^{(1)}-x_{2}\rvert \leq \frac{1}{2^n\sqrt{n}}$ and $\lvert x_{1}-x_{2}^{(1)}\rvert \geq \frac{1}{2^n}$, we have
\begin{align}\label{eq: lem2_2_num}
&\lvert x_{1}-w\rvert +\lvert w-x_{2}\rvert \notag\\
\leq& \lvert x_{1}-w\rvert +\lvert w-x_{2}^{(1)}\rvert +\lvert x_{2}^{(1)}-x_{2}\rvert  \notag\\
\leq& \lvert x_{1}-w\rvert +\lvert w-x_{2}^{(1)}\rvert +\tfrac{1}{\sqrt{N}}\lvert x_{1}-x_{2}^{(1)}\rvert  \notag\\
\leq& \big(1+\tfrac{1}{\sqrt{N}}\big)\big(\lvert x_{1}-w\rvert +\lvert w-x_{2}^{(1)}\rvert \big) \notag\\
\leq&\big(1+\tfrac{1}{\sqrt{N}}\big)\left(\big(1+\tfrac{1}{N}\big)\operatorname{Re}(w-x_{1})+\big(1+\tfrac{8}{\sqrt{N}}\big)\operatorname{Re}(x_{2}^{(1)}-w)\right) \notag\\
\leq&\big(1+\tfrac{1}{\sqrt{N}}\big)\big(1+\tfrac{8}{\sqrt{N}}\big)\operatorname{Re}(x_{2}^{(1)}-x_{1})\notag\\
\leq& \big(1+\tfrac{10}{\sqrt{N}}\big)\operatorname{Re}(x_{2}^{(1)}-x_{1}). 
\end{align}

On the other hand, 
\begin{align}\label{eq: lem2_de}
    \lvert x_{1}-x_{2}\rvert &\geq\lvert x_{1}-x_{2}^{(1)}\rvert -\lvert x_{2}-x_{2}^{(1)}\rvert  \geq\lvert x_{1}-x_{2}^{(1)}\rvert -\tfrac{1}{2^{n}\sqrt{n}} \notag\\
    &\geq\big(1-\tfrac{1}{\sqrt{N}}\big)\lvert x_{1}-x_{2}^{(1)}\rvert  \geq\big(1-\tfrac{1}{\sqrt{N}}\big)\operatorname{Re}(x_{2}^{(1)}-x_{1}).
\end{align}
From \eqref{eq: lem2_1_num}, \eqref{eq: lem2_2_num}, and \eqref{eq: lem2_de}, it follows that
\begin{align*}
    \max_{w\in\wideparen{x_1x_2}}&\frac{\lvert x_{1}-w\rvert +\lvert w-x_{2}\rvert }{\lvert x_{1}-x_{2}\rvert }\leq\tfrac{1+\tfrac{12}{\sqrt{N}}}{1-\tfrac{1}{\sqrt{N}}}<1+\tfrac{14}{\sqrt{N}}<1+\varepsilon.
\end{align*}

Now, we set 
\[
    \delta=\min\left\{\delta_1,\tfrac{1}{2^N}\right\}.
\] 
If $\lvert x_1-x_2\rvert <\delta$, $x_1,\ x_2$ must fall into one of the following cases:
\begin{itemize}
    \item $x_1, x_2 \in C_1\cup C_2\cup C_3\cup\left(\bigcup_{n=1}^{N+1}\gamma_n^{(n)}\right)$;
    \item $x_1, x_2 \in \bigcup_{n\geq N}\gamma_n^{(n)}$;
    \item $x_1\in C_1$, $x_2 \in \bigcup_{n\geq N}\gamma_n^{(n)}$.
\end{itemize}
Together with Lemma~\ref{x1x2_in_union_of_big_bumps}, the above analysis yields
\[
    \max_{w\in\wideparen{x_1x_2}}\frac{\lvert x_1-w\rvert +\lvert w-x_2\rvert }{\lvert x_1-x_2\rvert }<1+\varepsilon.
\]
This completes the proof.
\end{proof}

\begin{theorem}\label{chord-arc}
    The curve $\gamma$ constructed in Section~\ref{counterexample} is chord-arc. 
    In particular, this implies that $\gamma$ has no self-intersection, that is, $\gamma$ is a simple Jordan curve.
\end{theorem}
\begin{proof}
Notice that the curve $\gamma$ is a $C^1$ curve except at the point $O(0,0)$. 
To prove $\gamma$ is a chord-arc curve, it suffices to show that for any sufficiently small subarc $\wideparen{x_1x_2}$ containing $O(0,0)$, there exists a uniform constant $C$ such that
\[
    \frac{\ell\big(\wideparen{x_1x_2}\big)}{\lvert x_1-x_2\rvert }\leq C.
\]
By Theorem~\ref{asymp_conf}, $\gamma$ is asymptotically conformal. 
Taking $\varepsilon=1$, there exists $\delta>0$ such that whenever $\lvert x_1-x_2\rvert <\delta$,
\begin{align}\label{eq: ca_ac}
    \max_{w\in\wideparen{x_1x_2}}\frac{\lvert x_1-w\rvert +\lvert w-x_2\rvert }{\lvert x_1-x_2\rvert }<2.
\end{align}

Assume $x_1\in C_1$, $x_2\in\gamma_n^{(n)}$ with $\lvert x_1-x_2\rvert <\delta$ (if necessary, choosing a smaller $\delta$ so that $n$ could be sufficiently large satisfying all the preceding lemmas and claims). 
Let $x_2^{(k)}$ denote the projection of $x_2=x_2^{(n)}$ into $\gamma_n^{(k)}$ for $1\leq k\leq n$.

If $\lvert M_n-x_{2}^{(n-1)}\rvert <\frac{\varepsilon_{n}^{(n-1)}}{2}$, then $M_n$ and $x_{2}^{(n)}$ lie on the same bump of $\gamma_{n}^{(n)}$. 
By Lemma~\ref{x1x2_in_small_bump}, it follows that
\begin{align*}
    \frac{\ell\big(\wideparen{M_nx_2}\big)}{\lvert M_n-x_2\rvert }\leq 2.
\end{align*}
Otherwise, by the choice of $\varepsilon_n^{(j)}$, there must exist $1\leq j\leq n-1$ such that
\[
    \begin{cases}
       |M_n - x_2^{(j)}|\geq \tfrac{\varepsilon_n^{(j)}}{2}, \\
  |M_n - x_2^{(j-1)}|< \tfrac{\varepsilon  _n^{(j)}}{2}.
    \end{cases}
\]
Thus $M_n$ and $x_{2}^{(j)}$ lie on the same bump of $\gamma_{n}^{(j)}$, and by Lemma~\ref{x1x2_in_small_bump} again, 
\begin{align}\label{eq: ca_Mn_x2_j}
    \frac{\ell\big(\wideparen{M_nx_2^{(j)}}\big)}{\lvert M_n-x_2^{(j)}\rvert }\leq 2.
\end{align}
Note that if the first condition holds for $j=1$, we ignore the second condition in this case.

By Claim~\ref{projection} and inequality \eqref{eq: ca_Mn_x2_j}, we obtain
\begin{align}\label{eq: ca_Mn_x2_nu}
  \ell\big(\wideparen{M_n x_2}\big)
  &= \prod_{k=j+1}^n \frac{\ell\big(\wideparen{M_n x_2^{(k)}}\big)}{\ell\big(\wideparen{M_n x_2^{(k-1)}}\big)}\cdot \ell\big(\wideparen{M_n x_2^{(j)}}\big) \notag\\
  &\leq 2 \prod_{k=j+1}^n ( 1 + 8\beta_{n}^{(k)}) |M_n - x_2^{(j)}|\notag\\
  &= 2 \exp\left(\sum_{k=j+1}^n \log(1 + 8\beta_n^{(k)})\right) |M_n - x_2^{(j)}|\notag\\
  &\leq 2 \exp\left(8\sum_{k=j+1}^n \beta_n^{(k)}\right) |M_n - x_2^{(j)}|\leq 2e^8 |M_n - x_2^{(j)}|.
\end{align}
Moreover, together with Claim~\ref{D_gammak_to_gamman}, the condition $|M_n - x_2^{(j)}|\geq \tfrac{\varepsilon_n^{(j)}}{2}$ yields
\begin{align}\label{eq: ca_Mn_x2_de}
  |M_n - x_2|
  &\geq |M_n - x_2^{(j)}|- |x_2^{(j)} - x_2|\geq |M_n - x_2^{(j)}|- \tfrac{4}{\sqrt{n}} \varepsilon_n^{(j)} \notag\\
  &\geq \bigl(1 - \tfrac{8}{\sqrt{n}}\bigr)|M_n - x_2^{(j)}|\geq \tfrac{1}{2}|M_n - x_2^{(j)}|
\end{align}
for sufficiently large $n$.

As in the proof of Theorem~\ref{not_asymp_smooth}, by the construction process of $\gamma$ together with Lemma~\ref{lengthest_of_embeded_curve}, we deduce
\begin{align}\label{eq: ratiosup}
    \frac{\ell\bigl(\gamma_{n}^{(k)}\bigr)}{\ell\bigl(\gamma_{n}^{(k-1)}\bigr)}
    &=\frac{\sum_{i,\, m} \ell\bigl(\lambda_{i_m}^{(k-1)}\bigr)}{\sum_{i,\, m} \ell\bigl(\sigma_{i_m}^{(k-1)}\bigr)} \leq \sup_{i,\, m} \frac{\ell\bigl(\lambda_{i_{m}}^{(k-1)}\bigr)}{\ell\bigl(\sigma_{i_{m}}^{(k-1)}\bigr)}= \sup_{i,\, m} \ell\bigl(\Lambda_{i_{m}}^{(k-1)}\bigr)\notag\\
    &\leq\sup_{i}\bigl( 1+4\beta_{n}^{(k-1)}+\alpha_{i}^{(k-1)}\varepsilon_{n}^{(k-1)} K_{n}^{(k-1)}\sqrt{\beta_{n}^{(k-1)}}\bigr)\notag\\
    &\leq 1+4\beta_{n}^{(k-1)}+2\cdot\beta_{n}^{(k-1)}=1+6\beta_{n}^{(k-1)}
\end{align}
for sufficiently large $n$.
Therefore, using inequalities \eqref{eq: ratiosup} and \eqref{eq: length_est_gamma_n^1}, it follows that
\begin{align}\label{eq: ca_length_gamma_n}
    \ell\bigl(\gamma_n^{(n)}\bigr)
    &\leq\prod_{k=2}^n(1+6\beta_n^{(k-1)})\cdot\tfrac{1}{2 ^n}\bigl(1+\tfrac{4}{n^2}\bigr)\leq\tfrac{1}{2^{n-1}}\exp\left(6\sum_{k=2}^n\beta_n^{(k-1)}\right)\notag\\
    &=\tfrac{1}{2^{n-1}}\exp\left(\tfrac{6n(n-1)}{2n^2}\right)\leq\tfrac{1}{2^n}\cdot 2e^3.
\end{align}

Combining \eqref{eq: ca_length_gamma_n}, \eqref{eq: ca_Mn_x2_nu}, \eqref{eq: ca_Mn_x2_de} and the semicircle estimate, we conclude
\begin{align*}
  \ell\bigl(\wideparen{x_1 x_2}\bigr)
  &= \ell\bigl(\wideparen{x_1 O}\bigr) + \ell\bigl(\wideparen{O M_n}\bigr) + \ell\bigl(\wideparen{M_n x_2}\bigr) \\
  &\leq 2|x_1 - O|+ \sum_{k=n+1}^\infty \ell(\gamma_k^{(k)}) + 4 e^8 |M_n - x_2|\\
  &\leq 2|x_1 - O|+ 2 e^3 \sum_{k=n+1}^\infty \frac{1}{2^k} + 4 e^8 |M_n - x_2|\\
  &\leq 2|x_1 - O|+ 2 e^3 |O - M_n|+ 4 e^8 |M_n - x_2|\\
  &\leq 4 e^3 |x_1 - M_n|+ 4 e^8 |M_n - x_2|\leq 8e^8|x_1-x_2|,
\end{align*}
where the last line follows from \eqref{eq: ca_ac}. Hence $\gamma$ is chord-arc. 
\end{proof}

\section{Uniformly Approximable Chord-arc Curves}\label{charac}

In this section, we prove the following characterization of asymptotically smooth chord-arc curves by giving an additional property to asymptotic conformality.

\begin{theorem}\label{main2}
    A chord-arc curve $\Gamma$ is asymptotically smooth if and only if it is uniformly approximable and asymptotically conformal.
\end{theorem}

\begin{proof}
Suppose first that $\Gamma$ is asymptotically smooth. 
Then for every $\varepsilon>0$ there exists $\delta>0$ such that for any subarc $\wideparen{xy}$ with $|x-y|< \delta$ one has
\begin{align}\label{eq:sm1}
    \ell(\wideparen{xy}) \leq (1+\varepsilon)|x-y|.
\end{align}
Fix such an $\varepsilon>0$ and the corresponding $\delta>0$. 
Thus we can choose an integer $L\in\mathbb{N}$ such that
\begin{align*}
    \frac{\ell(\Gamma)}{L}< \delta.
\end{align*}
Let $\gamma$ be an arbitrary subarc of $\Gamma$ with endpoints $a$ and $b$. 
Partition $\gamma$ into $L$ subarcs $\wideparen{a_{i-1}a_i}$ $(i=1,\ldots, L)$ of equal length, where $a_0=a,\ a_L=b$. Then for each $i$  
\[
    |a_i-a_{i-1}|\leq \ell(\wideparen{a_{i-1}a_i})\leq \frac{\ell(\Gamma)}{L} < \delta,
\]
and thus, by \eqref{eq:sm1},
\begin{align}\label{eq:sm2}
\ell(\wideparen{a_{i-1}a_i}) \leq (1+\varepsilon)|a_i-a_{i-1}|.
\end{align}
Summing \eqref{eq:sm2} over $i$ yields
\begin{align*}
    \ell(\gamma)=\sum_{i=1}^L\ell(\wideparen{a_{i-1}a_i})\leq (1+\varepsilon)\sum_{i=1}^L|a_i-a_{i-1}|,
\end{align*}
which shows that $\Gamma$ is uniformly approximable. 
Moreover, asymptotic conformality follows immediately from the definition of asymptotic smoothness.

Conversely, assume that $\Gamma$ is uniformly approximable and asymptotically conformal. 
Suppose, for contradiction, that $\Gamma$ fails to be asymptotically smooth. 
Then there exists $\varepsilon_0>0$ such that for every $m\in\mathbb{N}$ one can find a subarc $\wideparen{a_m b_m}$ with
\begin{align}\label{eq:contradict}
    |a_m-b_m|\leq \frac{1}{m},\quad \frac{\ell(\wideparen{a_m b_m})}{|a_m-b_m|}\geq 1+\varepsilon_0.
\end{align}

Since $\Gamma$ is asymptotically conformal, for every $\varepsilon >0$ there exists $\delta>0$ such that for all $x,\,y\in\Gamma$ with $|x-y|< \delta$ and $w\in\wideparen{xy}$, 
\begin{align}\label{eq:conf}
    \max_{w\in \wideparen{xy}}\frac{|x-w|+|w-y|}{|x-y|}\leq 1+\varepsilon .
\end{align}
In addition, as $\Gamma$ is chord-arc, there exists $C\geq 1$ with
\begin{align}\label{eq:chordarc}
    \ell(\wideparen{xy}) \leq C|x-y|\qquad \text{for all }x,y\in\Gamma.
\end{align}

Fix $0<\varepsilon_1<\varepsilon_0$. 
By uniform approximability, there exists $L\in\mathbb{N}$ such that for any subarc $\wideparen{ab}$ there is a partition $a=a_0<a_1<\cdots<a_L=b$ for which
\begin{align}\label{eq:UA}
    \ell(\wideparen{ab}) \leq (1+\varepsilon_1)\sum_{i=1}^L |a_i-a_{i-1}|.
\end{align}
We take $m \in \mathbb{N}$ so large that $C/m \leq \delta$. 
For the subarc $\wideparen{ab}=\wideparen{a_m b_m}$ from \eqref{eq:contradict}, estimate \eqref{eq:chordarc} gives
\begin{align*}
    \ell(\wideparen{ab}) \leq C|a-b|\leq \frac{C}{m}\leq \delta,
\end{align*}
which implies that each chord joining any distinct two points $a_i, a_j$ in the partition $a_0<a_1<\cdots<a_L$ also satisfies 
\[
    |a_i-a_j|\leq \ell(\wideparen{a_ia_j})\leq \delta.
\]
Then \eqref{eq:conf} can be applied to the partition, and iterating its application yields
\begin{align}\label{eq:iterconf}
    \sum_{i=1}^L |a_i-a_{i-1}|&\leq (1+\varepsilon )\Bigl(|a_L-a_{L-2}|+\sum_{i=1}^{L-2}|a_i-a_{i-1}|\Bigr)\notag\\
    &\leq\cdots
    \leq (1+\varepsilon )^L |a-b|.
\end{align}

Finally, combining \eqref{eq:contradict}, \eqref{eq:UA}, and \eqref{eq:iterconf}, we obtain
\begin{align*}
    (1+\varepsilon_0)|a-b|
    &\leq \ell(\wideparen{ab})
    \leq (1+\varepsilon_1)\sum_{i=1}^L |a_i-a_{i-1}|\\
    &\leq (1+\varepsilon_1)(1+\varepsilon )^L |a-b|.
\end{align*}
Choosing $\varepsilon >0$ sufficiently small so that $(1+\varepsilon_1)(1+\varepsilon )^L<1+\varepsilon_0$ leads to a contradiction.

Therefore, $\Gamma$ must be asymptotically smooth, and the proof is complete.
\end{proof}
\begin{remark}
    Uniform approximability is a global property meaning that the length of any subarc of $\gamma$ can be uniformly approximated by the length of polygonal lines with a bounded number of vertices, independently of the choice of the subarc. 
    While asymptotic smoothness is a local property, ensuring that at arbitrarily small scales, the arc length approaches the chord length. 
\end{remark}


\section{Teichm\"uller Spaces}\label{Teich}
In this section, motivated by Theorem \ref{main1}, we introduce an intermediate Teichm\"uller space between the $\mathrm{BMO}$ and $\mathrm{VMO}$ Teichm\"uller spaces and discuss its properties and related problems (Theorem~\ref{intermediate}).
As preliminary background, we review the theory of Teichm\"uller spaces relevant to our arguments.
This section may also serve as preliminaries on holomorphic function spaces and spaces of circle homeomorphisms that appear in the main theorems stated in the introduction.
For fundamental results on quasiconformal mappings, see \cite{Ah}, and
for Teichm\"uller spaces, see \cite{Le} and \cite{N}.

\subsection{The Little Universal Teichm\"uller Space}
The space of Beltrami coefficients on the exterior unit disk $\mathbb{D}^*=\{|z|>1\} \cup \{\infty\}$ is defined as
\[
M(\mathbb{D}^*)=
\{\mu \in L^\infty(\mathbb{D}^*) :\Vert \mu \Vert_\infty<1\}.
\]
For $\mu \in M(\mathbb{D}^*)$, we denote by $f^\mu$ the normalized quasiconformal self-homeomorphism of $\mathbb{D}^*$.
It extends homeomorphically to the unit circle $\mathbb{S}$ as a quasisymmetric self-homeo\-morphism of $\mathbb{S}$, which we denote by the same symbol $f^\mu$. 
The normalization is imposed by fixing the three points $1$, $i$, and $-1$ on $\mathbb{S}$.

Two Beltrami coefficients $\mu, \nu \in M(\mathbb{D}^*)$ are said to be Teich\-m\"ul\-ler equivalent if $f^\mu$ and $f^\nu$ coincide on $\mathbb{S}$. 
The \emph{universal Teich\-m\"ul\-ler space} $T$ is the quotient of $M(\mathbb{D}^*)$ by Teich\-m\"ul\-ler equivalence. 
We denote this projection by $\pi\colon M(\mathbb{D}^*) \to T$ and write $\pi(\mu)=[\mu]$.

To endow $T$ with a complex structure, we introduce the complex Banach space $A(\mathbb{D})$ of holomorphic functions $\varphi$ on the unit disk $\mathbb{D} = \{z : |z|<1\}$ satisfying
\[
\Vert \psi \Vert_A=\sup_{|z|<1}\,(1-|z|^2)^2|\psi(z)|<\infty.
\]  
For every $\mu \in M(\mathbb{D}^*)$, let $f_\mu$ be the quasiconformal self-homeomorphism of the extended complex plane $\widehat {\mathbb{C}}$ with complex dilatation $0$ on $\mathbb{D}$ and $\mu$ on $\mathbb{D}^*$ , normalized by $f_\mu(0)=0$, $(f_\mu)'(0)=1$, and $f_\mu(\infty)=\infty$. 
Then we define the Bers projection $S$ by assigning to $\mu$ the Schwarzian derivative $S_{f_\mu}|_{\mathbb{D}}$ of $f_\mu|_{\mathbb{D}}$. 
It is well known that $S_{f_\mu}|_{\mathbb{D}} \in A(\mathbb{D})$ and that $S\colon M(\mathbb{D}^*) \to A(\mathbb{D})$ is a holomorphic split submersion onto its image (see \cite[Section~3.4]{N}).

Moreover, $S$ factors through $\pi$, yielding a well-defined injection $\alpha\colon T \to A(\mathbb{D})$ satisfying $\alpha \circ \pi = S$. 
This is called the \emph{Bers embedding} of $T$. 
By the properties of $S$, we conclude that $\alpha$ is a homeomorphism onto its image $\alpha(T) = S(M(\mathbb{D}^*))$, which is an open subset of
$A(\mathbb{D})$.
This provides $T$ with the structure of a complex Banach manifold as a bounded domain in $A(\mathbb{D})$.

The closed subspace $M_0(\mathbb{D}^*)$ of $M(\mathbb{D}^*)$ consisting of all asymptotically trivial elements is defined as
\[
M_0(\mathbb{D}^*)=\{\mu \in M(\mathbb{D}^*) \mid \lim_{t \to 0} \mathrm{ess\,sup}_{|z|<1+t} |\mu(z)|=0\}.
\]
The \emph{little universal Teichm\"uller space} is defined as $T_0=\pi(M_0(\mathbb{D}^*))$. 
For $\mu \in M_0(\mathbb{D}^*)$, the quasisymmetric self-homeomorphism $f^\mu$ of $\mathbb{S}$ is symmetric.
The set $\mathrm{S}(\mathbb{S})$ of all symmetric self-homeomorphisms of $\mathbb{S}$ is a subgroup of $\mathrm{QS}(\mathbb{S})$. In fact, it is the characteristic topological subgroup of $\mathrm{QS}(\mathbb{S})$. 
This and related results concerning $T_0$ were established by Gardiner and Sullivan \cite{GS}.

For the Banach space $A(\mathbb{D})$, we also define its little subspace as
\[
A_0(\mathbb{D})=\{\psi \in A(\mathbb{D}) \mid \lim_{t \to 0} \sup_{|z|>1-t}\,(1-|z|^2)^2|\psi(z)|=0\}.
\]
Then the restriction of $S$ to $M_0(\mathbb{D}^*)$ is a holomorphic split submersion onto its image in $A_0(\mathbb{D})$, and the restriction of $\alpha$ to $T_0$ is a homeomorphism onto $\alpha(T_0) = S(M_0(\mathbb{D}^*))$, which implies that $T_0$ is a complex Banach submanifold of $T$.
In fact, it is known that $\alpha(T_0)=\alpha(T) \cap A_0(\mathbb{D})$. 

For $\mu \in M(\mathbb{D}^*)$, recall that $f_\mu$ is the normalized quasiconformal self-homeomorphism of the extended complex plane $\widehat{\mathbb{C}}$ whose complex dilatation is $\mu$ on $\mathbb{D}^*$ and $0$ on $\mathbb{D}$.
Then $\varphi=\log(f_\mu|_{\mathbb{D}})'$ belongs to the space $B(\mathbb{D})$ of Bloch functions on $\mathbb{D}$, with the seminorm $\Vert \varphi \Vert_B=\sup_{z \in \mathbb{D}}(1-|z|^2)|\varphi'(z)|$ finite.  
Here, $\varphi'$ is the pre-Schwarzian derivative $P_{f_\mu}=(f_\mu)''/(f_\mu)'$ of $f_\mu$.  
By identifying Bloch functions that differ by a constant, we endow $B(\mathbb{D})$ as a Banach space equipped with the norm $\Vert \varphi \Vert_B$.  
We call the correspondence $L\colon M(\mathbb{D}^*) \to B(\mathbb{D})$, given by $\mu \mapsto \log(f_\mu|_{\mathbb{D}})'$, the \emph{pre-Bers projection}. 
Then $L$ is holomorphic, as in the case of $S$ (see \cite{Z} and \cite[Appendix]{Teo}).

We define a map $J\colon B(\mathbb{D}) \to A(\mathbb{D})$ by $\varphi \mapsto \psi=\varphi''-(\varphi')^2/2$, following the definition of the Schwarzian derivative.  
It is easy to see that $J$ is holomorphic and satisfies $S=J \circ L$.  
Therefore, $L$ is a holomorphic split submersion onto its image $\widetilde{\mathcal{T}}=L(M(\mathbb{D}^*))$ in $B(\mathbb{D})$. This space is biholomorphically equivalent to the Bers fiber space.
By restricting $J$ to $\widetilde{\mathcal{T}}$, again denoted by $J$, we obtain a holomorphic split submersion $J\colon\widetilde{\mathcal{T}} \to \alpha(T)$. 
In fact, $J$ is the projection of the real-analytic disk bundle $\widetilde{\mathcal{T}}$ over the universal Teichm\"uller space $\alpha(T) \cong T$.

We now consider the map $L$ restricted to $M(\mathbb{D}^*)$.
Its image is contained in the little Bloch space $B_0(\mathbb{D})$, and  
$\widetilde{\mathcal{T}}_0=L(M_0(\mathbb{D}))$ is a sub-bundle of $\widetilde{\mathcal{T}}$ over $\alpha(T_0) \cong T_0$.
In fact, we can verify that $\widetilde{\mathcal{T}}_0=\widetilde{\mathcal{T}} \cap B_0(\mathbb{D})$.

\subsection{The $\mathrm{BMO}$ and $\mathrm{VMO}$ Teichm\"uller Space}

The $\mathrm{BMO}$ Teichm\"uller space is introduced by Astala and Zinsmeister \cite{AZ}.
We define this and related spaces here.

A locally integrable function $f$ on $\mathbb{S}$ is said to have \emph{bounded mean oscillation} ($\mathrm{BMO}$), denoted by $f\in \mathrm{BMO}(\mathbb{S})$, if
\begin{equation*}
\|f\|_\mathrm{BMO}\coloneqq\sup_{I\subset \mathbb{S}}\frac{1}{|I|}\int_I |f(x)-f_{I}|\,|dx|< \infty,
\end{equation*}
where the supremum is taken over all intervals $I\subset \mathbb{S}$, $|I|$ denotes the Lebesgue measure of $I$, and $f_{I}=\tfrac{1}{|I|}\int_{I}f$ is the mean value of $f$ over $I$. 
The space $\mathrm{BMO}$, introduced by John and Nirenberg, is a Banach space modulo constants when equipped with the norm $\Vert \cdot \Vert_{BMO}$. 
We say $f$ has \emph{vanishing mean oscillation} ($\mathrm{VMO}$), denoted by $f\in \mathrm{VMO}(\mathbb{S})$, if $f\in \mathrm{BMO}(\mathbb{S})$ and 
\begin{equation*}
\lim_{|I|\to 0}\frac{1}{|I|}\int_I |f(x)-f_{I}|\,|dx|=0
\end{equation*}
uniformly over intervals $I$.

A Borel measure $m$ on $\mathbb{D}^*$ is called a \emph{Carleson measure} if
\[
\Vert m \Vert_c=\sup_{x \in \mathbb{S},\, r>0} \frac{m(\Delta(x,r) \cap \mathbb{D}^*)}{r}<\infty,
\]
where $\Delta(x,r)$ denotes the disk of radius $r>0$ centered at $x \in \mathbb{S}$. 
A Carleson measure $m$ is said to be \emph{vanishing} if
\[
\lim_{r \to 0} \frac{m(\Delta(x,r) \cap \mathbb{D}^*)}{r}=0
\]
uniformly in $x$. 
The set of Carleson measures on $\mathbb{D}^*$ is denoted by $\mathrm{CM}(\mathbb{D}^*)$, and the set of vanishing Carleson measures by $\mathrm{CM}_0(\mathbb{D}^*)$.  

We define the following subspaces of $M(\mathbb{D}^*)$, whose elements induce absolutely continuous Carleson measures:
\begin{align*}
M_B(\mathbb{D}^*)&=\{\mu \in M(\mathbb{D}^*) \mid dm_\mu=\tfrac{|\mu(z)|^2}{|z|^2-1}\,dxdy,\ m_\mu \in \mathrm{CM}(\mathbb{D}^*)\},\\
M_V(\mathbb{D}^*)&=\{\mu \in M(\mathbb{D}^*) \mid dm_\mu=\tfrac{|\mu(z)|^2}{|z|^2-1}\,dxdy,\ m_\mu \in \mathrm{CM}_0(\mathbb{D}^*)\}.
\end{align*}
The norm on $M_B(\mathbb{D}^*)$ is defined by $\Vert \mu \Vert_*=\Vert \mu \Vert_\infty +\Vert m_\mu \Vert_c^{1/2}$.  
Then $M_V(\mathbb{D}^*)$ is a closed subspace of $M_B(\mathbb{D}^*)$.

The $\mathrm{BMO}$ Teich\-m\"ul\-ler space $T_B$ is defined as $\pi(M_B(\mathbb{D}^*))$,  
and the $\mathrm{VMO}$ Teich\-m\"ul\-ler space $T_V$ as $\pi(M_V(\mathbb{D}^*))$. 

For $\mu \in M_B(\mathbb{D}^*)$, the associated quasisymmetric self-homeomorphism $f^\mu$ of $\mathbb{S}$ is strongly quasisymmetric, while for $\mu \in M_V(\mathbb{D}^*)$, $f^\mu$ is strongly symmetric (see \cite[Theorem 4]{AZ}, \cite[Theorem 4.1]{SW}).
Recall that $\mathrm{SQS}(\mathbb{S})$ is the group of all strongly quasisymmetric self-homeo\-morphisms of $\mathbb{S}$, and $\mathrm{SS}(\mathbb{S})$ is the subgroup consisting of all strongly symmetric self-homeomorphisms.
It is known from \cite[Lemma 5]{CF} that $\mathrm{SQS}(\mathbb{S})$ forms a subgroup of $\mathrm{QS}(\mathbb{S})$.
Moreover, $\mathrm{SS}(\mathbb{S})$ is the characteristic topological subgroup of $\mathrm{SQS}(\mathbb{S})$ (see \cite{We}).

Since $L^\infty(\mathbb{S})$ is dense in $\mathrm{VMO}(\mathbb{S})$ with respect to the $\mathrm{BMO}$ norm, we obtain the following result (see the proof of \cite[Lemma 3.3]{Sh}).

\begin{proposition}\label{SSinS}
A strongly symmetric self-homeomorphism of $\mathbb{S}$ is symmetric; that is, $\mathrm{SS}(\mathbb{S}) \subset \mathrm{S}(\mathbb{S})$. 
Hence, the $\mathrm{VMO}$ Teichm\"uller space $T_V$ is contained in $T_0$.
\end{proposition}

The corresponding Banach spaces of holomorphic functions on $\mathbb{D}$ associated with $M_B(\mathbb{D}^*)$ and $M_V(\mathbb{D}^*)$ are defined as follows:
\begin{align*}
A_B(\mathbb{D})&=\{\psi \in A(\mathbb{D}) \mid dm_\psi=(1-|z|^2)^3|\psi(z)|^2dxdy,\ m_\psi \in \mathrm{CM}(\mathbb{D})\},\\
A_V(\mathbb{D})&=\{\psi \in A(\mathbb{D}) \mid dm_\psi=(1-|z|^2)^3|\psi(z)|^2dxdy,\ m_\psi \in \mathrm{CM}_0(\mathbb{D})\}.
\end{align*}
The norm on $A_B(\mathbb{D})$ is defined by $\Vert \psi \Vert_{A_B} =\Vert m_\psi \Vert_c^{1/2}$.  
Then $A_V(\mathbb{D})$ is a closed subspace of $A_B(\mathbb{D})$.

The images of $M_B(\mathbb{D}^*)$ and $M_V(\mathbb{D}^*)$ under the Bers projection $S$ are contained in $A_B(\mathbb{D})$ and $A_V(\mathbb{D})$, respectively.
The following result was proved by Shen and Wei \cite[Theorem~5.1]{SW}.

\begin{proposition}\label{submersion}
The Bers projection $S\colon M_B(\mathbb{D}^*) \to A_B(\mathbb{D})$ is holomorphic and admits a local holomorphic right inverse $\sigma$ defined on some neighborhood $U$ of every point in the image $S(M_B(\mathbb{D}^*))$. 
Namely, $S \circ \sigma=\mathrm{id}|_U$. 
The restriction of $S$ to $M_V(\mathbb{D}^*)$ satisfies the analogous property as a holomorphic map into $A_V(\mathbb{D})$.
\end{proposition}

The Bers embeddings for $T_B$ and $T_V$ are obtained by restricting $\alpha$ to these spaces.  
Proposition~\ref{submersion} ensures that $\alpha\colon T_B \to A_B(\mathbb{D})$ and $\alpha\colon T_V \to A_V(\mathbb{D})$ are homeomorphisms onto their respective images, $\alpha(T_B)=S(M_B(\mathbb{D}^*))$ and $\alpha(T_V)=S(M_V(\mathbb{D}^*))$.  
The spaces $T_B$ and $T_V$ thus carry natural Banach manifold structures, with $T_V$ being a closed submanifold of $T_B$, and their Bers images satisfy 
\[
\alpha(T_B) = \alpha(T) \cap A_B(\mathbb{D}) \quad \text{and} \quad \alpha(T_V) = \alpha(T) \cap A_V(\mathbb{D}).
\]
Moreover, Proposition~\ref{SSinS} implies the inclusion 
$A_V(\mathbb{D}) \subset A_0(\mathbb{D})$.

The pre-Bers projections on $M_B(\mathbb{D}^*)$ and $M_V(\mathbb{D}^*)$ are similarly defined.  
The corresponding Banach spaces consist of  $\mathrm{BMOA}$ and $\mathrm{VMOA}$ functions on $\mathbb{D}$ (modulo additive constants):
\begin{align*}
\mathrm{BMOA}(\mathbb{D})&=\{\varphi \in B(\mathbb{D}) \mid dm_\varphi=(1-|z|^2)|\varphi'(z)|^2dxdy,\ m_\varphi \in \mathrm{CM}(\mathbb{D})\},\\
\mathrm{VMOA}(\mathbb{D})&=\{\varphi \in B(\mathbb{D}) \mid dm_\varphi=(1-|z|^2)|\varphi'(z)|^2dxdy,\ m_\varphi \in \mathrm{CM}_0(\mathbb{D})\}.
\end{align*}
The norm on $\mathrm{BMOA}(\mathbb{D})$ is defined by $\|\varphi\|_{\mathrm{BMOA}} = \|m_\varphi\|_c^{1/2}$, under which $\mathrm{VMOA}(\mathbb{D})$ forms a closed subspace. 
While several equivalent characterizations of $\mathrm{BMOA}$ are available (see \cite{Ga, Zhu}), we adopt the definition from \cite{AZ}; for a proof of equivalence, we refer to \cite[Theorem~6.5]{Gi}. 
As mentioned earlier, it is well known that $\mathrm{VMOA}(\mathbb{D}) \subset B_0(\mathbb{D})$, which may also be derived from Proposition~\ref{SSinS}.

The images of $M_B(\mathbb{D}^*)$ and $M_V(\mathbb{D}^*)$ under $L$, denoted by $\widetilde{\mathcal{T}}_B$ and $\widetilde{\mathcal{T}}_V$, are contained in $\mathrm{BMOA}(\mathbb{D})$ and $\mathrm{VMOA}(\mathbb{D})$, respectively. 
Moreover, we have $\widetilde{\mathcal{T}}_B=\widetilde{\mathcal{T}} \cap \mathrm{BMOA}(\mathbb{D})$ and $\widetilde{\mathcal{T}}_V=\widetilde{\mathcal{T}} \cap \mathrm{VMOA}(\mathbb{D})$ (see \cite[p.~143]{SW}).

We now consider the map $J$ on $\mathrm{BMOA}(\mathbb{D})$ and $\mathrm{VMOA}(\mathbb{D})$. 
The images of these spaces under $J$ lie in $A_B(\mathbb{D})$ and $A_V(\mathbb{D})$ respectively, and the map $J\colon\mathrm{BMOA}(\mathbb{D}) \to A_B(\mathbb{D})$ is holomorphic. 
From the relation $J \circ L = S$, we obtain a holomorphic surjection $J\colon \widetilde{\mathcal{T}}_B \to \alpha(T_B)$ satisfying $J(\widetilde{\mathcal{T}}_V) = \alpha(T_V)$. 
In fact, $J$ corresponds to the projection of the real-analytic disk bundle 
$\widetilde{\mathcal{T}}_B$ over the $\mathrm{BMO}$ Teichm\"uller space $\alpha(T_B) \cong T_B$, with $\widetilde{\mathcal{T}}_V$ forming a subbundle (see \cite{Ma}).

\subsection{An Intermediate Teichm\"uller Space Between $\mathrm{BMO}$ and $\mathrm{VMO}$}

Let \[M_{B_0}(\mathbb{D}^*) = M_{B}(\mathbb{D}^*) \cap M_{0}(\mathbb{D}^*)\]be a closed subspace of $M_{B}(\mathbb{D}^*)$ with the norm $\Vert \mu \Vert_{\ast}$, which satisfies $M_{V}(\mathbb{D}^*) \subset M_{B_0}(\mathbb{D}^*)\subset M_{B}(\mathbb{D}^*)$.
We define an intermediate Teichm\"uller space between $T_V$ and $T_B$
as $T_{B_0}=\pi(M_{B_0}(\mathbb{D}^*))$. 
For $\mu \in M_{B_0}(\mathbb{D}^*)$, the quasisymmetric self-homeomorphism $f^\mu$ lies in $\mathrm{SQS}(\mathbb{S}) \cap \mathrm{S}(\mathbb{S})$. 
The main theorem (Corollary~\ref{answer}) asserts that the inclusion $\mathrm{SS}(\mathbb{S}) \subset \mathrm{SQS}(\mathbb{S}) \cap \mathrm{S}(\mathbb{S})$ is strict.
Thus we obtain the strict inclusion of Teichm\"uller spaces $T_V \subset T_{B_0} \subset T_B$.

The corresponding Banach spaces of holomorphic functions are defined as $A_{B_0}(\mathbb{D})=A_B(\mathbb{D}) \cap A_0(\mathbb{D})$ with norm $\Vert \psi \Vert_{A_B}$ and $\mathrm{BMOA}_0(\mathbb{D})=\mathrm{BMOA}(\mathbb{D}) \cap B_0(\mathbb{D})$ with norm $\Vert \varphi \Vert_\mathrm{BMOA}$.
With these spaces, the properties of the previous Teichm\"uller spaces discussed above are also valid for this intermediate Teichm\"uller space $T_{B_0}$.

\begin{theorem}\label{intermediate}
The following statements hold:
\begin{enumerate}
\item The Bers projection $S\colon M_{B_0}(\mathbb{D}^*) \to A_{B_0}(\mathbb{D})$ is holomorphic and admits a local holomorphic right inverse at every point in the image $S(M_{B_0}(\mathbb{D}^*))$.
\item The Bers embedding $\alpha\colon T_{B_0} \to A_{B_0}(\mathbb{D})$ is a homeomorphism onto its image
$\alpha(T_{B_0})=S(M_{B_0}(\mathbb{D}^*))$, which coincides with $\alpha(T) \cap A_{B_0}(\mathbb{D})$.
\item The Teich\-m\"ul\-ler space $T_{B_0}$ is endowed with the structure of a complex Banach manifold
modeled on $A_{B_0}(\mathbb{D})$.
\item The pre-Bers projection $L\colon M_{B_0}(\mathbb{D}^*) \to \mathrm{BMOA}_0(\mathbb{D})$ is holomorphic and
its image $\widetilde {\mathcal{T}}_{B_0}$ coincides with $\widetilde {\mathcal{T}} \cap \mathrm{BMOA}_0(\mathbb{D})$.
\item The holomorphic map $J\colon\widetilde {\mathcal{T}}_{B_0} \to \alpha(T_{B_0})$ is the projection of
the real-analytic disk bundle $\widetilde {\mathcal{T}}_{B_0}$ over the Teichm\"uller space.
\end{enumerate}
\end{theorem}

\begin{proof}
The only claim requiring verification is the existence of a local holomorphic right inverse to $S$. 
The other statements follow directly from the arguments in the previous cases. 
We apply the local holomorphic right inverse $\sigma$ defined on $U$ in Proposition~\ref{submersion} also in the present setting. 
Then we can check that $\sigma$ maps $U \cap A_0(\mathbb{D})$ into $M_{B_0}(\mathbb{D}^*)$, which yields the desired local holomorphic right inverse to $S\colon M_{B_0}(\mathbb{D}^*) \to A_{B_0}(\mathbb{D})$.
\end{proof}

Finally, we consider $T_{B_0}$ and related spaces with a view toward whether the properties that hold for $T_V$ but fail for $T_B$ may still hold for $T_{B_0}$. 
Using equivalent characterizations of these spaces, we pose the following problems in this direction.

\begin{problem}
Is an asymptotically conformal quasicircle satisfying the Bishop--Jones condition necessarily chord-arc?
\end{problem}

Let $T_C$ be the set of elements in $T_B$ corresponding to chord-arc curves. 
It is known that $T_C$ is an open subset of $T_B$ containing $T_V$, but it is not known whether $T_C$ is connected (see \cite{AZ}). 
The above problem asks whether $T_C$ contains $T_{B_0}$.

\begin{problem}
Is $L^\infty(\mathbb{S})$ dense in $\mathrm{BMO}_0(\mathbb{S})$, the boundary extension of $\mathrm{BMOA}_0(\mathbb{D})$?  
Is $e^\varphi$ an $A_2$-weight for every $\varphi \in \mathrm{BMO}_0(\mathbb{S})$?
\end{problem}

As mentioned above, $L^\infty(\mathbb{S})$ (more precisely, its intersection with $\mathrm{VMO}(\mathbb{S})$) is dense in $\mathrm{VMO}(\mathbb{S})$, but not in $\mathrm{BMO}(\mathbb{S})$. 
The denseness of $L^\infty(\mathbb{S})$ is an important property for the study of subspaces of the $\mathrm{BMO}$ Teichm\"uller space.
For $\varphi \in \mathrm{BMO}(\mathbb{S})$, the distance from $\varphi$ to $L^\infty(\mathbb{S})$ in the $\mathrm{BMO}$ norm is comparable to the infimum of such $t>0$ for which $e^{t\varphi}$ is an $A_2$-weight, by a theorem of Garnett and Jones (see \cite[Corollary IV.5.15]{GR}).

\begin{problem}
Does there exist a real-analytic global section of the Teichm\"uller projection onto $T_{B_0}$?
\end{problem}

Using the barycentric extension due to Douady and Earle \cite{DE}, a global section $s_\mathrm{DE}\colon T_B \cong \mathrm{SQS}(\mathbb{S})/\mathrm{M\ddot ob}(\mathbb{S}) \to M_B(\mathbb{D}^*)$ satisfying $\pi \circ s_\mathrm{DE}=\mathrm{id}|_{T_B}$, was constructed by Cui and Zinsmeister \cite{CZ}.
It was shown in \cite{TWS} that $s_\mathrm{DE}$ is continuous when restricted to $T_V$, from which the contractibility of $T_V$ follows. 
On the other hand, using a variant of the Beurling--Ahlfors extension via the heat kernel (see \cite{FKP}, \cite{WM-1}), a real-analytic global section $s_\mathrm{FKP}\colon T_V \cong \mathrm{SS}(\mathbb{S})/\mathrm{M\ddot ob}(\mathbb{S})\to M_V(\mathbb{D}^*)$ was constructed in \cite[Theorem 6.4]{WM-2}. 
The above problem asks whether an analogous construction is possible for $T_{B_0}$.
For the construction of $s_\mathrm{FKP}$, the denseness of $L^\infty(\mathbb{S})$ plays a crucial role.

\bigskip
\noindent
{\bf Conflict of interest.} 
On behalf of all authors, the corresponding author states that there is no conflict of interest regarding the publication of this paper.

\end{document}